\documentclass[a4paper,11pt]{article}
\pdfoutput=1
\usepackage{hyperref}
\hypersetup{hypertexnames = false, bookmarksdepth = 2, bookmarksopen = true, colorlinks, linkcolor = black, citecolor = black, urlcolor = black, pdfstartview={XYZ null null 1}}

\usepackage{amsfonts}
\usepackage[fleqn, leqno]{amsmath}
\usepackage{amsthm}

\usepackage[capitalise]{cleveref}

\usepackage[utf8]{inputenc}
\usepackage[backend=biber, maxbibnames=99, datamodel=mrnumber, sortcites]{biblatex}

\usepackage{booktabs}
\usepackage{colonequals}
\usepackage{diagbox}
\usepackage{enumitem}
\usepackage{fullpage}
\usepackage{mathtools}
\usepackage{parskip}
\usepackage{thmtools}
\usepackage{tikz-cd}
\usepackage{xparse}
\usepackage{xspace}

\usepackage[T1]{fontenc}
\usepackage{libertine}
\usepackage[libertine]{newtxmath}
\usepackage[scaled=0.83]{beramono}
\usepackage{eucal}
\usepackage{microtype}
\usepackage{gitinfo2}

\DeclareFieldFormat{mrnumber}{%
  MR\addcolon\space
  \ifhyperref
    {\href{http://www.ams.org/mathscinet-getitem?mr=1#1}{\nolinkurl{#1}}}
    {\nolinkurl{#1}}}

\renewbibmacro*{doi+eprint+url}{%
  \iftoggle{bbx:doi}
    {\printfield{doi}}
    {}%
  \newunit\newblock
  \printfield{mrnumber}%
  \newunit\newblock
  \iftoggle{bbx:eprint}
    {\usebibmacro{eprint}}
    {}%
  \newunit\newblock
  \iftoggle{bbx:url}
    {\usebibmacro{url+urldate}}
    {}}

\declaretheoremstyle[spaceabove = 3pt, spacebelow = 3pt, bodyfont = \itshape]{theorem}
\declaretheoremstyle[spaceabove = 3pt, spacebelow = 3pt]{remark}

\declaretheorem[style=theorem]{theorem}
\declaretheorem[style=theorem, sibling=theorem]{corollary}
\declaretheorem[style=theorem, sibling=theorem]{lemma}
\declaretheorem[style=theorem, sibling=theorem]{proposition}

\declaretheorem[style=remark, sibling=theorem]{definition}
\declaretheorem[style=remark, sibling=theorem]{example}

\declaretheorem[style=remark, sibling=theorem]{remark}

\declaretheorem[style=theorem, numberwithin=section, title=Theorem]{alphatheorem}
\declaretheorem[style=theorem, sibling=alphatheorem, title=Conjecture]{alphaconjecture}

\declaretheorem[style=theorem, sibling=alphatheorem, title=Proposition]{alphaproposition}

\crefname{alphatheorem}{Theorem}{Theorems}
\crefname{alphaconjecture}{Conjecture}{Conjectures}
\crefname{alphacorollary}{Corollary}{Corollaries}
\crefname{alphaproposition}{Proposition}{Propositions}
\crefname{alphaquestion}{Question}{Questions}

\crefformat{enumi}{#2\textup{(#1)}#3}

\setlist[enumerate]{font=\normalfont}

\mathchardef\mhyphen="2D
\newcommand\dash{\nobreakdash-\hspace{0pt}}

\DeclareMathOperator\Alb{Alb}
\DeclareMathOperator\alb{alb}
\DeclareMathOperator\Aut{Aut}
\DeclareMathOperator\Bihol{Bihol}
\DeclareMathOperator\derived{\mathbf{D}}
\DeclareMathOperator\HH{H}
\DeclareMathOperator\hh{h}
\DeclareMathOperator\Perf{Perf}
\DeclareMathOperator\Pic{Pic}

\newcommand\bounded{\ensuremath{\mathrm{b}}}

\title{The Albanese morphism for hyperelliptic varieties}
\author{Pieter Belmans \and Andreas Demleitner \and Pedro N\'u\~nez}

\begin{document}
\maketitle


\begin{abstract}
  We explicitly describe the Albanese morphism
  of a hyperelliptic variety,
  i.e., the quotient~$X$ of an abelian variety~$A$ by a finite group~$G$
  acting freely and not only by translations,
  by giving a description of the Albanese variety
  and the Albanese fibers in terms of~$A$ and~$G$.
  In particular, the fibers are themselves abelian or hyperelliptic varieties,
  and we investigate which can occur in explicit examples.
  As an application
  we show that the derived category of~$X$ is indecomposable in certain cases.
\end{abstract}


\section{Introduction}
Hyperelliptic varieties are smooth projective varieties
defined as quotients of abelian varieties
by finite groups acting freely and not only via translations.
In particular, they are smooth projective varieties which admit abelian varieties
as finite \'etale covers,
and thus they are again of Kodaira dimension~0,
without being abelian varieties themselves.

Given a hyperelliptic variety $X=A/G$, we will study the associated Albanese morphism
\begin{equation}
  \label{equation:albanese}
  \alb_X\colon X\to\Alb(X)
\end{equation}
where~$\Alb(X)$ is the Albanese variety,
an abelian variety whose dimension is
given by the \emph{irregularity}~$q_X = \hh^1(X,\mathcal{O}_{X})$.
By \cite[Theorem~8.3]{MR814013}
the Albanese morphism~$\alb_X$ of a variety with torsion canonical bundle
is an \'etale fiber bundle with smooth and connected fibers,
i.e.,
there exists an \'etale covering~$B\to\Alb(X)$
such that~$X\times_{\Alb(X)}B\cong F\times B$,
where~$F$ is a smooth projective variety which is a fiber of~$\alb_X$,
and moreover the Kodaira dimension of~$F$ is~0 by \cref{remark:torsion-order-of-fiber}.

Our main result is an explicit and effective description of
both~$\Alb(X)$
and the Albanese fiber~$F$
in terms of the defining data~$A$ and~$G$ for~$X$.
\begin{alphatheorem}
  \label{theorem:fiber-albanese}
  Let~$X=A/G$ be a hyperelliptic variety.
  Then
  \begin{enumerate}
    \item $\Alb(X)$ is an explicit quotient of a subquotient of~$A$
      (see \eqref{equation:candidate-albanese} for the precise form)
    \item the fibers of the (isotrivial) Albanese morphism \eqref{equation:albanese}
      are explicit subquotients of~$A$
      (see \cref{proposition:fibers} for the precise form),
      so they are either abelian varieties or hyperelliptic varieties.
  \end{enumerate}
\end{alphatheorem}
Here, subquotient means it is the quotient of an abelian subvariety of~$A$ by some group action induced by~$G$.
In case of~$\Alb(X)$, the subquotient is necessarily an abelian variety,
whereas for the fibers it depends on the geometry.
The more precise version is given in \Cref{theorem:albanese-morphism}.

It turns out that a precursor to this result
can be found by inspecting the details of the proof of \cite[Theorem~7.13]{MR360582},
which is concerned with understanding quotients of abelian varieties by arbitrary group actions.
However, in the more special case of hyperelliptic varieties,
the resulting description of the Albanese is more explicit.
See \cref{remark:existing} for more context.

With the explicit description given by \cref{theorem:fiber-albanese} in hand,
we start the study of \emph{which} (abelian or hyperelliptic) varieties
can arise as Albanese fibers
in \cref{section:albanese-fibers}.
This becomes an important ingredient in
the geography of hyperelliptic varieties.

\paragraph{Indecomposability of the derived category}
Understanding (in)decomposability of the derived category of coherent sheaves
is an important step in understanding the derived category itself.
By Serre duality,
Calabi--Yau varieties have an indecomposable derived category.
For curves
it was shown in \cite{MR2838062}
that the only curve~$C$ for which~$\derived^\bounded(C)$ is decomposable
is~$\mathbb{P}^1$.

Already for surfaces the situation is more interesting.
By Okawa \cite[Theorem 1.7]{2304.14048}
the derived category of a smooth projective minimal surface~$S$
with~$\HH^1(S,\mathcal{O}_S)\neq 0$ is indecomposable.
For minimal surfaces of general type with~$\HH^2(S,\mathcal{O}_S)\neq 0$
(which are conjectured to have indecomposable derived categories)
a starting point is \cite[\S4.3]{1508.00682v2}.
By Orlov's blowup formula, if~$S$ is not minimal,
then~$\derived^\bounded(S)$ is decomposable.

Surfaces of Kodaira dimension~0 come in 4~types:
abelian, bielliptic, Enriques and K3.
The first and the last have trivial canonical bundle,
thus their derived categories are indecomposable.
The other two classes have non-trivial torsion canonical bundle.
On Enriques surfaces every line bundle is an exceptional object,
and there are interesting semiorthogonal decompositions,
whose study was initiated in \cite{MR1629793}.
For bielliptic surfaces however
the derived category is indecomposable by the earlier cited result of Okawa.

Generalizing the indecomposability for bielliptic surfaces
we propose the following conjecture,
and we will subsequently use the properties of the Albanese morphism
to partially settle it.
\begin{alphaconjecture}
  \label{conjecture:indecomposability}
  Let~$X$ be a hyperelliptic variety.
  Then~$\derived^\bounded(X)$ is indecomposable.
\end{alphaconjecture}

As evidence for this conjecture,
we will prove in \cref{section:indecomposability}
the following.
\begin{alphaproposition}
  \label{proposition:three-results}
  Let~$X=A/G$ be a hyperelliptic variety.
  Then~$\derived^\bounded(X)$ is indecomposable in the following cases:
  \begin{enumerate}
    \item $X$ is cyclic, i.e., $G$ is cyclic,
      in which case $X$ is even stably indecomposable;
    \item $X$ has irregularity~$q_X=\dim X-2$ or~$\dim X-1$,
      in which case $X$ is even stably indecomposable;
    \item the Albanese fiber has trivial canonical bundle.
  \end{enumerate}
  In particular, all hyperelliptic 3-folds have an indecomposable derived category.
\end{alphaproposition}
The notion of stable indecomposability for a variety
is a strengthening of that of indecomposability for its derived category.
It is introduced in \cite{MR4582884},
and we will recall it in \cref{definition:stably-indecomposable}.

These three points are proven in
\cref{corollary:cyclic,corollary:big-irregularity,corollary:cy-fibers}
respectively.
The first two build upon a theorem of Pirozhkov
(see \cref{theorem:pirozhkov}),
whilst the latter builds upon an indecomposability criterion of Kawatani--Okawa
(see \cref{theorem:kawatani-okawa}).

The conditions in \cref{proposition:three-results} are not mutually disjoint:
by \cref{proposition:large-irregularity-is-cyclic},
any variety with irregularity~$q_X=\dim X-1$ is cyclic,
and all three conditions apply to bielliptic surfaces.
In higher dimensions however they diverge,
as explained in \cref{remark:conditions-diverge}.

One class of examples which is \emph{not} covered by any of these methods
are the regular hyperelliptic 4-folds where~$G=\mathrm{D}_4\times\mathbb{Z}/2\mathbb{Z}$
appearing in \cite[\S10.1.2]{2211.07998}:
they have~$q_X=0$ and~$\omega_X\neq\mathcal{O}_X$, as~$\mathrm{h}^{0,4}(X)=0$.
The methods involving the paracanonical base locus \cite[Corollary~1.6]{2107.09564v3}
also do not apply to this example,
because the paracanonical base locus equals the canonical base locus
(as~$\Pic^0X$ is trivial)
and the base locus of~$\omega_X$ is~$X$,
as it is non-trivial but torsion.

\paragraph{Conventions}
Throughout we work over~$\mathbb{C}$.

\paragraph{Acknowledgements}
The second named author extends thanks to Julia Kotonski for valuable exchanges.
The third named author would like to thank Jungkai Alfred Chen for useful discussions.
The third named author was supported by the NSTC of Taiwan, with grant numbers 112-2811-M-002-108 and 113-2811-M-002-087.

\section{Preliminaries}
\label{section:preliminaries}
In this section we recall some basics
and set up the notation and terminology.

\paragraph{Hyperelliptic varieties}
The classification of K-trivial surfaces
has~4 types, one of which are the bielliptic surfaces.
These have been classified by Bagnera--de Franchis in 1908 and Enriques--Severi in 1909,
and they are usually presented as
the surfaces of the form~$(E_1\times E_2)/G$
where~$E_1,E_2$ are elliptic curves,
and~$G$ is a finite group which acts on~$E_1$ not only by translations,
and on~$E_2$ by translations.
The study of their higher-dimensional cousins was initiated by Lange in \cite{MR1862215}.
We will define them and recall some notation and elementary properties.

\begin{definition}
  \label{definition:hyperelliptic}
  A \emph{hyperelliptic manifold} is the quotient $X =  T/G$ of
  a complex torus $T = \mathbb{C}^{n}/\Lambda$ by the action of
  a finite subgroup $G \subseteq \operatorname{Bihol}(T)$
  that acts freely on $T$ and is not a subgroup of the group of translations.
  A projective hyperelliptic manifold is called a \emph{hyperelliptic variety}.
\end{definition}

We will write~$r=r_X\colon T\twoheadrightarrow X=T/G$ for the \'etale quotient map,
or~$r=r_X\colon A\twoheadrightarrow X=A/G$ when~$X$ is a variety
(which happens precisely when the torus is an abelian variety).
Since $X$ is a finite \'etale quotient of a complex torus, its canonical divisor is torsion,
so $X$ has Kodaira dimension~0.

By \cite[Theorem~1.3]{MR4363769} the study of smooth quotients~$T/G$
where~$G$ is any finite group of biholomorphisms
is reduced to fibrations in products of projective spaces
over tori and hyperelliptic manifolds.
Thus the objects in \cref{definition:hyperelliptic} are the key
to understand such quotients.

\begin{remark}
  \label{remark:no-translations}
  Without loss of generality,
  we may assume that the subgroup $G$ in \Cref{definition:hyperelliptic}
  contains no translations \cite[page~306]{MR3404712}.
  Indeed, the subgroup of translations in $G$ is a normal subgroup $H \subseteq G$,
  so we can write
  \begin{equation}
    T/G \cong (T/H)/(G/H),
  \end{equation}
  where $T/H$ is again a complex torus
  and $G/H$ is a non-trivial subgroup of automorphisms
  that acts freely and contains no translations.
  From now on, we will frequently make this assumption without further notice.
\end{remark}
The following object associated to a hyperelliptic manifold will be important.
\begin{definition}
  Let $X = T/G$ be a hyperelliptic manifold of dimension~$n=\dim X$,
  where $T = \mathbb{C}^{n}/\Lambda$.
  Then we denote by
  \begin{equation}
    \label{equation:complex-representation}
    \tilde{\rho}_X \colon G \to \operatorname{GL}(n,\mathbb{C})
  \end{equation}
  the restriction of the complex representation~%
  $\operatorname{Bihol}(T) \to \operatorname{GL}(n,\mathbb{C})$,
  and call it the \emph{complex representation of $X$}.
  Moreover, we denote by
  \begin{equation}
    \label{equation:representation}
    \rho_{X} \colon G \to \Bihol_{0}(T)
  \end{equation}
  the representation that sends an automorphism $g \in G$ to the Lie group automorphism $\rho_{X}(g)$
  whose lift to the universal cover is $\tilde{\rho}_{X}(g)$, i.e., for all $g \in G$ we have
  \begin{equation}
    \label{equation:explicit-representation}
    \rho_{X}(g) = g - g(0).
  \end{equation}
\end{definition}

\begin{remark}
  \label{remark:complex-representation}
  If we assume as in \Cref{remark:no-translations} that $G$ does not contain any translations,
  then the complex representation~$\tilde{\rho}_X$ is \emph{faithful} \cite[Remark~3.2.2(1)]{2211.07998},
  i.e.,~it is injective.
  In this case, it is also the holonomy representation of the compact flat Riemannian manifold underlying~$X$,
  see \cite[Remark~3.2.2(2)]{2211.07998}.
  In particular, if we assume that $G$ does not contain any translations,
  then~$G$ is the holonomy group of~$X$,
  hence it is an intrinsic invariant of $X$.

  The complex torus $T$ can also be recovered from $X$,
  because its defining lattice is the unique maximal normal abelian subgroup of $\pi_{1}(X)$,
  fitting into a short exact sequence
  \begin{equation}
    \label{equation:fundamental-groups}
    1 \to \Lambda = \pi_{1}(T) \to \pi_{1}(X) \to G \to 1,
  \end{equation}
  see \cite[\S 5.2]{MR3404712}.
  If $X$ is projective, then GAGA implies that $T$ admits
  exactly one algebraic structure compatible with its complex structure,
  so in this case the abelian variety $A$ can also be recovered from~$X$.
\end{remark}

We are interested in the Albanese variety, whose dimension is given by the irregularity.
The assumption that $G$ does not act only via translations in \Cref{definition:hyperelliptic}
implies the following bound on the irregularity.
\begin{lemma}
  \label{lemma:irregularity}
  Let $X = T/G$ be a hyperelliptic manifold and let $q_{X}$ be its irregularity.
  Then
  \begin{equation}
    \label{equation:irregularity-bound}
    q_{X} < \dim(X).
  \end{equation}
\end{lemma}

\begin{proof}
  Since $T$ is a complex torus of the same dimension as $X$,
  and since $G$ acts on $T$ not only via translations, we have
  \begin{equation}
    \label{equation:irregularity}
    q_X
    =\hh^{0}(X,\Omega_{X}^{1})
    =\dim \HH^{0}(T,\Omega_{T}^{1})^{G}
    <\hh^{0}(T,\Omega_{T}^{1})=\dim T=\dim X,
  \end{equation}
  which proves the claim.
\end{proof}

\begin{remark}
  \label{remark:irregularity}
  It follows from the above proof that the irregularity $q_X$ of a hyperelliptic manifold $X = T/G$, where $T = V/\Lambda$,  equals
  \begin{equation}
    q_X = \dim(V^\vee)^G = \dim V^G,
  \end{equation}
  which in turn equals the multiplicity of the trivial representation in the complex representation $\tilde{\rho}_X$.
\end{remark}

An important class of hyperelliptic manifolds is the following.
\begin{definition}
  \label{definition:bdf}
  A \emph{cyclic} hyperelliptic manifold is a hyperelliptic manifold $X = T/G$ such that $G$ is a cyclic group.
\end{definition}
Cyclic hyperelliptic manifolds are called \emph{Bagnera--de Franchis manifolds} in \cite[\S 4]{MR3404712}.
The following result shows that hyperelliptic manifolds
with the largest possible irregularity are examples of cyclic hyperelliptic manifolds.
One recipe to classify cyclic hyperelliptic varieties,
viable in low dimensions,
is given in \cite[\S5]{MR1862215}.
A more general recipe is given in \cite{MR4809689}.

\begin{proposition}
  \label{proposition:large-irregularity-is-cyclic}
  Let $X=T/G$ be a hyperelliptic manifold of irregularity~$q_X=\dim X-1$.
  Then~$X$ is cyclic.
\end{proposition}

\begin{proof}
  Writing~$T=V/\Lambda$,
  we have that~$\HH^0(X,\Omega_X^1)=\HH^0(T,\Omega_T^1)^G=(V^\vee)^G$.
  Thus under our assumption on the irregularity,
  the complex representation $\tilde{\rho}_{X}$ contains exactly one non-trivial $1$-dimensional representation~$\chi \colon G \to \mathbb{C}^{\times}$.
  Since $\tilde{\rho}_{X}$ is faithful, $\chi$ has to be faithful, i.e., $\chi$ is injective.
  This means that $G \cong \chi(G)$ is a finite subgroup of $\mathbb{C}^{\times}$, hence a cyclic group.
\end{proof}

\begin{remark}
  \label{remark:low-irregularity-cyclic}
  The converse of \Cref{proposition:large-irregularity-is-cyclic} is not true.
  Indeed, let $n \geq 2$, let $T$ be a product of $n$ elliptic curves,
  let $E$ be any elliptic curve, and consider the $\mathbb{Z}/2\mathbb{Z}$-action
  on $T \times E$ given by
  \begin{equation}
    (a,z) \mapsto \left(-a, z + t \right), \qquad t \in E[2] \setminus \{0\}.
  \end{equation}
  Then $X \colonequals (T \times E)/(\mathbb{Z}/2\mathbb{Z})$ is a cyclic hyperelliptic variety
  of dimension $\dim{X} = n + 1$ and of irregularity~$q_{X} = 1$,
  so that~$q_{X} = \dim{X} - n < \dim{X} - 1$.
\end{remark}

\begin{remark}
  \label{remark:cyclic-irregularity-bound}
  It is however true that if $X$ is a cyclic hyperelliptic variety,
  then its irregularity is at least $1$, see \cite[Lemma 3.3]{MR1862215}.
\end{remark}

Usually, bielliptic surfaces (to which \cref{proposition:large-irregularity-is-cyclic} applies)
are \emph{not} written using cyclic groups.
Of the~7 deformation families in \cite[\S VI.20]{MR1406314},
3~use non-cyclic groups.
However our standing assumption is that~$G$ contains \emph{no} translations,
which one can assume as in \cref{remark:no-translations}.
One is referred to \cite[Example~3.2.4]{2211.07998} for the description
of bielliptic surfaces using this assumption,
where cyclicity is thus visible.

This description of bielliptic surfaces can also be used to find an example
showing that the converse of \Cref{proposition:large-irregularity-is-cyclic} is not true.

\begin{example}
  \label{example:small-irregularity-cyclic}
  Let $S_{1}$ be a bielliptic surface with holonomy group~$\mathbb{Z}/2\mathbb{Z}$,
  and let~$S_{2}$ be a bielliptic surface with holonomy group~$\mathbb{Z}/3\mathbb{Z}$,
  which exist by \cite[Example~3.2.4]{2211.07998}.
  By \Cref{remark:complex-representation}
  we can write~$S_{1} = A_{1}/(\mathbb{Z}/2\mathbb{Z})$ and~$S_{2} = A_{2}/(\mathbb{Z}/3\mathbb{Z})$,
  with both groups acting freely and without translations on the abelian surfaces~$A_1$ and~$A_2$.
  Then, the cyclic group $\mathbb{Z}/6\mathbb{Z}$ acts freely and without translations
  on the abelian variety $A_{1} \times A_{2}$,
  so $X\colonequals (A_{1} \times A_{2})/(\mathbb{Z}/6\mathbb{Z})$ is a cyclic hyperelliptic variety
  of dimension $4$ and irregularity $2$.
\end{example}

Complex tori have irregularity equal to their dimension, so \Cref{lemma:irregularity} implies
that hyperelliptic manifolds cannot be isomorphic to complex tori.
The following result, which is mentioned without a proof as \cite[Remark 3.2.2(3)]{2211.07998},
shows that they cannot even be homeomorphic.
\begin{lemma}
  \label{lemma:notabelian}
  A hyperelliptic manifold has non-abelian fundamental group, and thus cannot be homeomorphic to a Lie group.
\end{lemma}

\begin{proof}
  Let~$X = T/G$ be a hyperelliptic manifold and let~$\tilde{\rho}_X$ be its complex representation.
  By assumption we can find some $g \in G$
  such that $\tilde{g}(z) = \tilde{\rho}_{X}(g)(z) + b$ with
  $\tilde{\rho}_{X}(g) \neq \operatorname{id}_{\mathbb{C}^{n}}$,
  where $\tilde{g} \colon \mathbb{C}^{n} \to \mathbb{C}^{n}$
  denotes the lift of the automorphism $g$ to the universal cover of $X$.

  Let $\Lambda \subseteq \mathbb{C}^{n}$ be a lattice such that $T = \mathbb{C}^{n}/\Lambda$,
  and let $\lambda \in \Lambda$ be a vector such that $\tilde{\rho}_{X}(g)(\lambda) \neq \lambda$,
  which must exist because $\Lambda$ spans $\mathbb{C}^{n}$,
  and $\tilde{\rho}_{X}(g) \neq \operatorname{id}_{\mathbb{C}^{n}}$ by assumption.
  Then $\tilde{h}(z)\colonequals z + \lambda$ is a lift of the identity on $T$.

  Both $\tilde{g}$ and $\tilde{h}$ are elements of the group of deck transformations
  of the universal cover of $X$, corresponding to elements of the fundamental group of $X$,
  so it suffices to show that they do not commute.
  But for~$z \in \mathbb{C}^{n}$ we have
  \begin{equation}
    \tilde{g}(\tilde{h}(z))
    =
    \tilde{\rho}_{X}(g)(z) + \tilde{\rho}_{X}(g)(\lambda) + b \neq \tilde{\rho}_{X}(g)(z) + \lambda + b
    =
    \tilde{h}(\tilde{g}(z)),
  \end{equation}
  because $\lambda$ was chosen so that $\tilde{\rho}_{X}(g)(\lambda) \neq \lambda$.
  Therefore, the fundamental group of $X$ is non-abelian.

  On the other hand, the fundamental group of a Lie group is always abelian.
  So a hyperelliptic manifold cannot be homeomorphic to a Lie group.
\end{proof}

\begin{remark}
  \label{remark:notabelian-birational}
  In fact, one can show that a hyperelliptic variety cannot be birationally equivalent to an abelian variety.
  Indeed, a birational equivalence between a hyperelliptic variety and an abelian variety
  can be lifted to a generically \'etale morphism between abelian varieties,
  and such a morphism is necessarily everywhere \'etale.
  The original birational equivalence must have finite fibers,
  hence be an isomorphism by Zariski's Main Theorem, contradicting \Cref{lemma:notabelian}.
  Alternatively, one can again invoke \Cref{lemma:irregularity} and use the fact that the irregularity is a birational invariant.
\end{remark}

%
%

\paragraph{The Albanese morphism}
Recall from \cite[Theorem V.13]{MR1406314} that,
given a smooth projective variety $X$,
there exists an abelian variety~$\Alb(X)$ (the \emph{Albanese variety})
together with a morphism
\begin{equation}
  \operatorname{alb}_{X} \colon X \to \Alb(X)
\end{equation}
(the \emph{Albanese morphism})
characterized by the following universal property:~%
for any complex torus $T$ and any morphism $f \colon X \to T$
there exists a unique morphism $\tilde{f} \colon \Alb(X) \to T$
such that $\tilde{f} \circ \operatorname{alb}_{X} = f$.
More explicitly, integration along paths gives a morphism
\begin{equation}
  \pi_{1}(X) \to \HH^{0}(X,\Omega_{X}^{1})^{\vee}.
\end{equation}
We can construct the Albanese variety as
\begin{equation}
  \label{equation:albanese-variety}
  \Alb(X) = \HH^{0}(X,\Omega_{X}^{1})^{\vee}/\operatorname{im}(\pi_{1}(X)),
\end{equation}
and the morphism $\operatorname{Alb}_{X} \colon X \to \Alb(X)$ is induced by integration of differential forms along paths.
The dimension of the Albanese variety is given by the irregularity~$q_X=\hh^0(X,\Omega_X^1)=\hh^1(X,\mathcal{O}_X)$.

As already mentioned in the introduction,
the Albanese morphism for a hyperelliptic variety has many pleasant properties,
which we encode in the following lemma.
\begin{lemma}
  \label{lemma:fiber-bundle-structure}
  Let~$X$ be a hyperelliptic variety.
  Then~$\alb_X$ is an \'etale fiber bundle with smooth and connected fibers of Kodaira dimension~0,
  i.e.,
  there exists an \'etale covering~$B\to\Alb(X)$
  such that~$X\times_{\Alb(X)}B\cong F\times B$,
  where~$F$ is a smooth projective variety of Kodaira dimension~0.
\end{lemma}

\begin{proof}
  By \cite[Theorem~8.3]{MR814013},
  the Albanese morphism~$\alb_X$ of a variety with a torsion canonical bundle
  is an \'etale fiber bundle,
  i.e., there exists an \'etale covering $B \to \Alb(X)$
  such that~$X\times_{\Alb(X)}B \cong F\times B$,
  where~$F$ is a fiber of~$\alb_X$.
  By \cite[Theorem~1]{MR622451} the fiber $F$ is connected.
  By generic smoothness,~$F$ is moreover smooth.
  To see that~$\kappa(F) = 0$,
  note that~$F$ has trivial normal bundle in $X$
  because it is the fiber over a point,
  so by the adjunction formula we have~$\mathrm{K}_{X}|_{F} = \mathrm{K}_{F}$
  (see also \cref{remark:torsion-order-of-fiber}).
  Since~$\mathrm{K}_{X}$ is torsion, so is $\mathrm{K}_{F}$, and we conclude that $\kappa(F) = 0$.

  Alternatively,
  because a hyperelliptic variety is the quotient of an abelian variety,
  one can use \cite[Theorem~7.13 and Remark~7.16]{MR360582} to conclude.
\end{proof}

\section{On the Albanese morphism}
\label{section:albanese-morphism}
Let~$X=A/G$ be a hyperelliptic variety,
where~$A=V/\Lambda$ is an abelian variety,
and~$G\subseteq\Aut(A)$ a finite subgroup of automorphisms
containing no translations.
Recall that we write~$r\colon A\twoheadrightarrow X$
for the \'etale quotient map.
Consider the representation $\rho_X \colon G \to \Aut_{0}(A)$,
which maps $g \in G$ to its linear part,
i.e.,
to the algebraic group automorphism $\rho_{X}(g) = g - g(0)$.

\paragraph{Decomposing the abelian variety $A$}
First we will construct two abelian subvarieties~$A_0,A_1\subset A$,
suitably compatible with the action of~$G$,
which will be used to describe the Albanese variety,
resp.~the Albanese fibers.
This decomposition will be a partial \emph{isotypical decomposition}
of $A$ with respect to the $\rho_{X}(G)$-action, cf.~\cite[\S 13.6]{MR2062673}.

Define
\begin{equation}
  \label{equation:subvariety}
  A_0 \colonequals \left(\bigcap_{g \in G} \mathrm{ker}(\rho_{X}(g) - \mathrm{id}_A)\right)^0,
\end{equation}
where the superscript zero indicates that we take the connected component of the origin.

Note that $A_{0} \subseteq A$ is an abelian subvariety,
so we may write it as $A_{0} = V_{0}/\Lambda_{0}$
for some linear subspace~$V_{0} \subseteq V$,
and~$\Lambda_{0}\colonequals \Lambda \cap V_{0}$.
By definition, the induced linear action of $G$ on $A_0 = V_0/\Lambda_0$ is trivial.

Let $A_1 \subseteq A$ be an abelian subvariety which is stable under the $\rho_{X}(G)$-action and such that the addition map
\begin{equation}
  \label{equation:mu}
  \mu \colon A_0 \times A_1 \to A:(a_0,a_1) \mapsto a_0 + a_1
\end{equation}
is an isogeny.
Note that this isogeny is then automatically $\rho_{X}(G)$-equivariant.

One can construct such a subvariety as follows.
Choose an ample line bundle $L$ on $A$.
By averaging, we may choose it to be $G$-invariant.
Its first Chern class is represented by an alternating form
\begin{equation}
  E \colon V \times V \to \mathbb{R}
\end{equation}
such that $E(\Lambda,\Lambda) \subseteq \mathbb{Z}$ and $E(iv,iw) = E(v,w)$ for all $v, w \in V$
\cite[Proposition 2.1.6]{MR2062673},
and this alternating form is the imaginary part of a positive-definite hermitian form
(cf.~\S4 of op.~cit)
\begin{equation}
  H \colon V \times V \to \mathbb C
\end{equation}
whose imaginary part maps $\Lambda \times \Lambda$ to $\mathbb Z$ \cite[Lemma 2.1.7]{MR2062673}.
Note that by construction $H$ is $G$-invariant as well
(which is the same as being~$\tilde{\rho}_{X}(G)$-invariant,
as $\mathrm{c}_1(L)$ is translation-invariant).
Define $V_1$ as the orthogonal complement of $V_0$ in $V$ with respect to $H$ and let~$\Lambda_1\colonequals \Lambda \cap V_1$.

\begin{lemma}
  \label{lemma:linear-action}
  The subvariety
  \begin{equation}
    \label{equation:second-subvariety}
    A_1 \colonequals V_1/\Lambda_1
  \end{equation}
  is an abelian subvariety of~$A$ which is stable under the $\rho_{X}(G)$-action,
  such that the addition map \eqref{equation:mu} is an isogeny.
\end{lemma}

\begin{proof}
  The fact that the addition map \eqref{equation:mu} is an isogeny follows from
  Poincar\'e's complete reducibility theorem \cite[Theorem~5.3.5]{MR2062673}.
  To see that $A_1$ is stable under the $\rho_{X}(G)$-action, let $\tilde{\rho}_{X}(g) \in \operatorname{GL}(V)$
  be the lift of the complex representation to the universal cover.
  Then, for any $v_0 \in V_0$ and $v_1 \in V_1$, since $\tilde{\rho}_{X}(G)$ acts trivially on $V_{0}$ and $H$ is $\tilde{\rho}_{X}(G)$-invariant, we have
  \begin{equation}
    H(v_{0},\tilde{\rho}_{X}(g)(v_{1})) = H(\tilde{\rho}_{X}(g)(v_0), \tilde{\rho}_{X}(g)(v_1)) = H(v_0, v_1) = 0.
  \end{equation}
  Hence $\tilde{\rho}_{X}(g)(v_{1}) \in V_{1}$ and $A_1$ is indeed stable under the $\rho_{X}(G)$-action.
\end{proof}

For $g \in G$, there exists a unique $\tau(g) \in A$ such that
\begin{equation}
  \label{equation:tau}
  g(-) = \rho_{X}(g)(-) + \tau(g).
\end{equation}
Surjectivity of the addition map~$\mu$
in \eqref{equation:mu}
implies that we may write
\begin{equation}
  \label{equation:taudecomposition}
  \tau(g) = t_g^{(0)} + t_g^{(1)},
\end{equation}
where $t_g^{(i)} \in A_i$ for each $i \in \{0,1\}$.
This decomposition is however not unique.

Note that for $g, h \in G$ we have
\begin{equation}
  \label{equation:tau-cocycle}
  \tau(gh) = \rho_{X}(g)(\tau(h)) + \tau(g).
\end{equation}
The function $\tau \colon G \to A$ defines a~1-cocycle,
and the role of the cohomology class~$[\tau] \in\HH^{1}(G,A)$
in the theory of hyperelliptic varieties is described in \cite[Proposition 2.3]{MR1862215}.

\paragraph{Describing the $G$-action using the decomposition}
Now we express the~$G$-action on~$A$
in terms of~$A_0$ and~$A_1$.

\begin{lemma}
  \label{lemma:description-of-action}
  Let $K \colonequals \ker(\mu)$.
  Then the isogeny $\mu$ induces an isomorphism
  \begin{equation}
    \label{equation:barmu}
    \bar{\mu} \colon (A_{0} \times A_{1})/K \cong A,
  \end{equation}
  and the $G$-action on $A$ induces a unique $G$-action on $(A_{0} \times A_{1})/K$
  making the isomorphism $\bar{\mu}$ equivariant with respect to the $G$-action.
  This $G$-action can be written as
  \begin{equation}
    g \cdot ((a_{0},a_{1}) + K) = (a_{0} + t_{g}^{(0)}, \rho_{X}(g)(a_{1}) + t_{g}^{(1)}) + K
  \end{equation}
  for all $g \in G$.
\end{lemma}

\begin{proof}
  Direct computation.
  One can also check directly that the given formula
  defines a $G$-action on the left-hand side
  without using the $G$-action on the right-hand side,
  using that $\tau(hg) = \rho_{X}(h)(\tau(g)) + \tau(h)$ for all $g,h \in G$.
\end{proof}

\begin{remark}
  Under the isomorphism $A_{0} \times A_{1} \cong V/(\Lambda_{0} \oplus \Lambda_{1})$, we have
  \begin{equation}
    K \cong \Lambda /(\Lambda_{0} \oplus \Lambda_{1}).
  \end{equation}
\end{remark}

\begin{remark}
  \label{remark:subgroup}
  Let $K_i \subset A_i$ be the image of $K$
  under the projection $p_{i} \colon A_0 \times A_1 \to A_i$ for $i \in \{0, 1\}$.
  Then~$K_0 \cong K \cong K_1$, because the restriction $p_{i}|_{K} \colon K \to A_{i}$ is also injective.
\end{remark}

\paragraph{Constructing the candidate Albanese variety}
The next step is to construct what will turn out to be the Albanese variety,
by considering a quotient of~$A_0$.

\begin{lemma}
  \label{lemma:action}
  With the notation above, the group $G$ acts on the abelian variety $A_{0}/K_{0}$ via
  \begin{equation}
    g \cdot (a_{0} + K_{0}) = a_{0} + t_{g}^{(0)} + K_{0},
  \end{equation}
  and the resulting action is independent of the choice of~$t_{g}^{(0)}$.
\end{lemma}

\begin{proof}
  For $e \in G$, we have $\tau(e) = 0$ in $A$, so $t_{e}^{(0)} + t_{e}^{(1)} = 0$ in $A$.
  Therefore, $(t_{e}^{(0)},t_{e}^{(1)}) \in K$ and $t_{e}^{(0)} \in K_{0}$,
  and this implies that
  \begin{equation}
    e \cdot (a_{0} + K_{0}) = a_{0} + K_{0}
  \end{equation}
  for all $a_{0} \in A_{0}$.

  Let now $g, h \in G$ and let $a_{0} \in A_{0}$.
  We want to show that $(gh) \cdot (a_{0} + K_{0}) = g \cdot (h \cdot (a_{0} + K_{0}))$,
  so we need to show that
  \begin{equation}
    t_{gh}^{(0)} - t_{g}^{(0)} - t_{h}^{(0)} \in K_{0}.
  \end{equation}
  Using \eqref{equation:taudecomposition} and \eqref{equation:tau-cocycle}
  we get
  \begin{equation}
    t_{gh}^{(0)} + t_{gh}^{(1)} - \rho_{X}(g)(t_{h}^{(0)}) - \rho_{X}(g)(t_{h}^{(1)}) - t_{g}^{(0)} - t_{g}^{(1)} = 0.
  \end{equation}
  But $\rho_{X}(g)(t_{h}^{(0)}) = t_{h}^{(0)}$,
  because $t_{h}^{(0)} \in A_{0}$ by definition,
  and $\rho_{X}(g)(t_{h}^{(1)}) \in A_{1}$ by \Cref{lemma:description-of-action}.
  Therefore,
  \begin{equation}
    (t_{gh}^{(0)} - t_{h}^{(0)} - t_{g}^{(0)}, t_{gh}^{(1)} - \rho_{X}(g)(t_{h}^{(1)}) - t_{g}^{(1)}) \in K
  \end{equation}
  and $t_{gh}^{(0)} - t_{h}^{(0)} - t_{g}^{(0)} \in K_{0}$.

  For the independence of the choice of~$t_{g}^{(0)}$'s,
  note that if $t_{g}^{(0)} + t_{g}^{(1)} = s_{g}^{(0)} + s_{g}^{(1)}$
  with $t_{g}^{(i)}, s_{g}^{(i)} \in A_{i}$,
  then $(t_{g}^{(0)} - s_{g}^{(0)}, t_{g}^{(1)} - s_{g}^{(1)}) \in K$
  and $t_{g}^{(0)} - s_{g}^{(0)} \in K_{0}$.
\end{proof}

We can thus consider the quotient
\begin{equation}
  \label{equation:candidate-albanese}
  B\colonequals (A_0/K_0)/G.
\end{equation}

\begin{lemma}
  \label{lemma:candidate-albanese-variety-is-abelian}
  The quotient~$B$ in \eqref{equation:candidate-albanese} is an abelian variety.
\end{lemma}

\begin{proof}
  This follows because~$A_{0}/K_{0}$ is an abelian variety
  and $G$ acts on $A_{0}/K_{0}$ via translations,
  by \cref{lemma:action}.
\end{proof}

\paragraph{Constructing the candidate Albanese morphism}
Having constructed a candidate for the Albanese variety in \eqref{equation:candidate-albanese}
we will now construct a candidate for the Albanese morphism.

\begin{lemma}
  \label{lemma:albanese-morphism}
  There is a surjective morphism
  \begin{equation}
    \label{equation:albanese-morphism-to-be}
    f \colon X \to B
  \end{equation}
  fitting into the commutative diagram
  \begin{equation}
    \begin{tikzcd}
      A_{0} \times A_{1} \arrow[two heads, swap]{d}{s} \arrow[two heads]{r}{p_{0}}               & A_{0} \arrow[two heads]{d}{s_{0}} \\
      (A_{0} \times A_{1})/K \arrow{d}{\cong}[swap]{\bar{\mu}} \arrow[two heads]{r}{\bar{p}_{0}} & A_{0}/K_{0} \arrow[equal]{d} \\
      A \arrow[two heads, swap]{d}{r} \arrow[two heads]{r}{\bar{p}_{0} \circ \bar{\mu}^{-1}}     & A_{0}/K_{0} \arrow[two heads]{d}{r_{0}} \\
      A/G \arrow[two heads]{r}{f}                                                                & (A_{0}/K_{0})/G.
    \end{tikzcd}
  \end{equation}
\end{lemma}

\begin{proof}
  The projection $(A_0 \times A_1)/K \to A_0/K_0$ induces a surjective morphism
  \begin{equation}
    A \cong (A_0 \times A_1)/K \to B.
  \end{equation}
  This morphism takes constant values on $G$-orbits,
  and thus factors through the quotient by~$G$, which is~$X$.
  This establishes the existence of \eqref{equation:albanese-morphism-to-be}.
  Its role in the commutative diagram is immediate from the construction.
\end{proof}

\paragraph{Determining the candidate Albanese fibers}
Next, we will describe the fibers of \eqref{equation:albanese-morphism-to-be}.
\begin{lemma}
  \label{lemma:explicit-fibers}
  For all $a_{0} \in A_{0}$, we have
  \begin{equation}
    \label{equation:explicit-fibers}
    f^{-1}([a_{0}+K_{0}]_{G}) = r(\{a_{0}\} + A_{1}) \subseteq A/G.
  \end{equation}
\end{lemma}

\begin{proof}
  Let $a_{0} \in A_{0}$ and consider $[a_{0} + K_{0}]_{G} \in (A_{0}/K_{0})/G = B$.
  Since all the vertical arrows on the left-hand side of
  the commutative diagram in \Cref{lemma:albanese-morphism} are surjective, we have
  \begin{equation}
    f^{-1}([a_{0}+K_{0}]_{G})
    =
    r(\bar{\mu}(s(p_{0}^{-1}s_{0}^{-1}r_{0}^{-1}([a_{0} + K_{0}]_{G})))).
  \end{equation}
  By definition of the group action from \Cref{lemma:action}, we have
  \begin{equation}
    r_{0}^{-1}([a_{0}+K_{0}]_{G})
    =
    \bigcup_{g \in G} \{ a_{0} + t_{g}^{(0)} + K_{0} \} \subseteq A_{0}/K_{0}.
  \end{equation}
  Therefore, we have
  \begin{equation}
    s_{0}^{-1}r_{0}^{-1}([a_{0}+K_{0}]_{G})
    =
    \bigcup_{g \in G} \bigcup_{k_{0} \in K_{0}} \{ a_{0} + t_{g}^{(0)} + k_{0} \} \subseteq A_{0}
  \end{equation}
  and
  \begin{equation}
    p_{0}^{-1}s_{0}^{-1}r_{0}^{-1}([a_{0}+K_{0}]_{G})
    =
    \bigcup_{g \in G} \bigcup_{k_{0} \in K_{0}} \{ a_{0} + t_{g}^{(0)} + k_{0} \} \times A_{1} \subseteq A_{0} \times A_{1}
  \end{equation}
  We claim that
  \begin{equation}
    s(p_{0}^{-1}s_{0}^{-1}r_{0}^{-1}([a_{0}+K_{0}]_{G}))
    =
    \bigcup_{g \in G} \{ a_{0} + t_{g}^{(0)} \} \times A_{1} + K.
  \end{equation}
  Indeed, let $g \in G$, let $k_{0} \in K_{0}$ and let $a_{1} \in A_{1}$.
  Then, there exists some $k_{1} \in K_{1} \subseteq A_{1}$ such that $(k_{0},k_{1}) \in K$, so we can write
  \begin{equation}
    (a_{0} + t_{g}^{(0)} + k_{0} , a_{1}) + K
    =
    (a_{0} + t_{g}^{(0)}, a_{1} - k_{1}) + K \in \{ a_{0} + t_{g}^{(0)} \} \times A_{1} + K.
  \end{equation}
  Next, we claim that
  \begin{equation}
    r(\bar{\mu}(s(p_{0}^{-1}s_{0}^{-1}r_{0}^{-1}([a_{0}+K_{0}]_{G}))))
    =
    [\{ a_{0} \} + A_{1}]_{G} = \{ [a_{0} + a_{1}]_{G} \mid a_{1} \in A_{1} \} \subseteq A/G. \end{equation}
  Indeed, let $g \in G$ and let $a_{1} \in A_{1}$.
  Since $\tau(e) = 0$, \eqref{equation:tau-cocycle} implies that
  \begin{equation}
    \tau(g^{-1})
    =
    -\rho_{X}(g^{-1})(\tau(g))
    =
    -\rho_{X}(g^{-1})(t_{g}^{(1)}) - t_{g}^{(0)},
  \end{equation}
  so we can write
  \begin{equation}
    [a_{0} + t_{g}^{(0)} + a_{1}]_{G}
    =
    [g^{-1}(a_{0} + t_{g}^{(0)} + a_{1})]_{G}
    =
    [a_{0} + \rho_{X}(g^{-1})(a_{1}) - \rho_{X}(g^{-1})(t_{g}^{(1)})]_{G},
  \end{equation}
  and since $\rho_{X}(g^{-1})(a_{1}) - \rho_{X}(g^{-1})(t_{g}^{(1)}) \in A_{1}$ by \Cref{lemma:linear-action},
  we see that this element is in $[\{a_{0}\} + A_{1}]_{G}$.
  In conclusion, we have shown that
  \begin{equation}
    f^{-1}([a_{0}+K_{0}]_{G}) = \{ [a_{0} + a_{1}]_{G} \mid a_{1} \in A_{1} \} \subseteq A/G,
  \end{equation}
  and the right-hand side is precisely $r(\{a_{0}\} + A_{1}) \subseteq A/G$.
\end{proof}

Next, we describe a subgroup of~$G$
which will act on (translates of)~$A_1$
to give a succinct description of the fibers in \cref{proposition:fibers}.
\begin{lemma}
  \label{lemma:subgroup}
  With the notation above, define
  \begin{equation}
    \label{equation:subgroup}
    H \colonequals \{ g \in G \mid t_{g}^{(0)} \in K_{0} \}.
  \end{equation}
  Then $H \subseteq G$ is a well-defined subgroup.
\end{lemma}

\begin{proof}
  We show first that $H$ is well-defined as a subset of~$G$,
  i.e., that it is independent of the choice of~$t_{g}^{(0)}$ for~$g\in G$.
  For every $g \in G$, $\tau(g)\in A$ was uniquely determined by the property in \eqref{equation:tau}.
  Suppose now that we have
  \begin{equation}
    \tau(g) = t_{g}^{(0)} + t_{g}^{(1)} = s_{g}^{(0)} + s_{g}^{(1)},
  \end{equation}
  where $t_{g}^{(0)}, s_{g}^{(0)} \in A_{0}$ and $t_{g}^{(1)}, s_{g}^{(1)} \in A_{1}$.
  The claim is that $t_{g}^{(0)} \in K_{0}$
  if and only if $s_{g}^{(0)}\in K_0$.

  By symmetry, it suffices to show that if $t_{g}^{(0)} \in K_{0}$,
  then $s_{g}^{(0)} \in K_{0}$.
  Assume that $t_{g}^{(0)} \in K_{0}$.
  This means that there exists some $a_{1} \in A_{1}$ such that $(t_{g}^{(0)},a_{1}) \in K$,
  i.e., such that $t_{g}^{(0)} + a_{1} = 0$ in $A$.
  We can write
  \begin{equation}
    s_{g}^{(0)} + s_{g}^{(1)} - t_{g}^{(0)} - t_{g}^{(1)} = 0.
  \end{equation}
  Therefore, substituting $t_{g}^{(0)} = -a_{1}$, we see that
  \begin{equation}
    s_{g}^{(0)} + (s_{g}^{(1)} + a_{1} - t_{g}^{(1)}) = 0,
  \end{equation}
  and the right term of the sum is in $A_{1}$.
  Hence, $s_{g}^{(0)} \in K_{0}$ as well.

  Let us now check that $H \subseteq G$ is a subgroup.
  For $e \in G$, we have $\tau(e) = 0$.
  Therefore, we can choose~$t_{e}^{(0)} = 0 \in K_{0}$, and thus $e \in H$.

  Suppose now that $g, h \in H$.
  We want to show that $hg \in H$ as well.
  We have
  \begin{equation}
    \tau(hg) = \rho_{X}(h)(\tau(g)) + \tau(h).
  \end{equation}
  Let us write
  \begin{equation}
    \begin{aligned}
      \tau(g) &= t_{g}^{(0)} + t_{g}^{(1)},
      \tau(h) &= t_{h}^{(0)} + t_{h}^{(1)}
    \end{aligned}
  \end{equation}
  with $t_{g}^{(0)},t_{h}^{(0)} \in A_{0}$
  and $t_{g}^{(1)},t_{h}^{(1)} \in A_{1}$.
  Then we can write
  \begin{equation}
    \begin{aligned}
      \tau(hg)
      &= \rho_{X}(h)(t_{g}^{(0)} + t_{g}^{(1)}) + t_{h}^{(0)} + t_{h}^{(1)} \\
      &= (t_{g}^{(0)} + t_{h}^{(0)}) + (\rho_{X}(h)(t_{g}^{(1)}) + t_{h}^{(1)}),
    \end{aligned}
  \end{equation}
  and since $\rho_{X}(h)(t_{g}^{(1)}) \in A_{1}$,
  the left term is in $A_{0}$ and the right term is in $A_{1}$.
  Therefore, since $K_{0}$ is a subgroup,
  we deduce that the left term is in $K_{0}$, and thus $hg \in H$.
\end{proof}

We now come to the description of the fibers of \eqref{equation:albanese-morphism-to-be}
as quotients.

\begin{proposition}
  \label{proposition:fibers}
  For each $a_{0} \in A_{0}$, we have
  \begin{equation}
    \label{equation:fibers}
    f^{-1}([a_{0}+K_{0}]_{G}) = (\{ a_{0} \}+A_{1})/H.
  \end{equation}
  In particular, the fibers of $f \colon X \to B$ are either abelian varieties
  or hyperelliptic varieties.
\end{proposition}

\begin{proof}
  We have seen in \Cref{lemma:explicit-fibers} that~$f^{-1}([a_{0} + K_{0}]_{G})=r(\{ a_{0} \} + A_{1}) \subseteq A/G$.
  We consider now the induced surjective morphism
  \begin{equation}
    r|_{\{a_{0}\} + A_{1}} \colon \{a_{0}\} + A_{1} \to f^{-1}([a_{0} + K_{0}]_{G}).
  \end{equation}
  The subvariety $\{ a_{0} \} + A_{1}$ is invariant under the action of the subgroup $H$,
  because for all $a_{1} \in A_{1}$ and all $h \in H$ we can find some $k_{1} \in K_{1}$
  such that $t_{h}^{(0)} = -k_{1}$, hence
  \begin{equation}
    h(a_{0} + a_{1})
    =
    a_{0} + t_{h}^{(0)} + \rho_{X}(h)(a_{1}) + t_{h}^{(1)}
    =
    a_{0} - k_{1} +  \rho_{X}(h)(a_{1}) + t_{h}^{(1)} \in \{ a_{0} \} + A_{1}.
  \end{equation}
  Moreover, the restriction $r|_{\{a_{0}\} + A_{1}}$ takes constant values on the $H$-orbits,
  because $r$ is just the geometric quotient for the $G$-action on $A$ and $H \subseteq G$.
  Therefore, we obtain a surjective morphism
  \begin{equation}
    \bar{r} \colon (\{ a_{0} \} + A_{1})/ H \to f^{-1}([a_{0} + K_{0}]_{G}).
  \end{equation}
  Let us show that this morphism is also injective.
  By definition, we have
  \begin{equation}
    r^{-1}(f^{-1}([a_{0}+K_{0}]_{G}))
    =
    r^{-1}(r(\{a_{0}\} + A_{1}))
    =
    \{ g (a_{0}+a_{1}) \mid a_{1} \in A_{1}, g \in G \}.
  \end{equation}
  For each $g \in G$, since $\rho_{X}(g)(a_{0}) = a_{0}$ by definition of~$A_0$,
  $\rho_{X}(g)(a_{1}) \in A_{1}$ by \cref{lemma:linear-action},
  and $t_{g}^{(1)} \in A_{1}$ by construction, we have
  \begin{equation}
    g(a_{0}+a_{1})
    =
    a_{0} + a_{1}' + t_{g}^{(0)}
  \end{equation}
  for some $a_{1}' \in A_{1}$.
  From this we see that an element $g(a_{0}+a_{1}) \in r^{-1}(r(\{a_{0}\}+ A_{1}))$ is in $\{a_{0}\} + A_{1}$
  if and only if $t_{g}^{(0)} \in K_{0}$, i.e., if and only if $g \in H$.
  Therefore,
  \begin{equation}
    \bar{r} \colon (\{a_{0}\} + A_{1})/H \to f^{-1}([a_{0} + K_{0}]_{G})
  \end{equation}
  is an isomorphism.

  Finally, note that $G$ acts freely on $A$, so $H \subseteq G$ acts freely on $\{a_{0}\} + A_{1}$.
  If every automorphism in $H$ acts as a translation on the abelian variety $\{a_{0}\} + A_{1}$,
  then $(\{a_{0}\} + A_{1})/H$ is also an abelian variety.
  Otherwise, $H$ will contain automorphisms which are not translations,
  and $(\{a_{0}\} + A_{1})/H$ is a hyperelliptic variety.
\end{proof}

\paragraph{Verifying that the candidate is the Albanese}
We now arrive at the final step of the proof.

\begin{lemma}
  \label{lemma:factorization}
  Let $\operatorname{alb}_{X} \colon X \to \Alb(X)$ be the Albanese morphism of $X$.
  The morphism $f \colon X \to B$ in \eqref{equation:albanese-morphism-to-be} factors as
  \begin{equation}
    \begin{tikzcd}
      X \arrow{r}{\operatorname{alb}_{X}} \arrow[swap, rd, "f"] & \Alb(X) \arrow{d}{\bar{f}} \\
      & B,
    \end{tikzcd}
  \end{equation}
  where $\bar{f}$ is an isogeny.
\end{lemma}

\begin{proof}
  The universal property of the Albanese yields the factorization in the statement.
  Recall from \eqref{equation:albanese-variety} that~$\Alb(X) \cong \HH^0(X,\Omega^1_X)^\vee/\mathrm{im}(\pi_{1}(X))$.
  Moreover, $\HH^0(X,\Omega^1_X) \cong \HH^0(A,\Omega^1_A)^G$,
  which is isomorphic to the dual of $V_0$.
  It follows that
  \begin{equation}
    \dim \Alb(X) = \dim{V_{0}} = \dim B.
  \end{equation}
  Since $f$ is surjective, so is $\bar{f}$.
  Hence, $\bar{f}$ is an isogeny.
\end{proof}

We now come to the proof of the main result.
Recall that we defined~$A_0$ in \eqref{equation:subvariety},
and~$A_1$ in \eqref{equation:second-subvariety}.
The group~$K_0$ acting on~$A_0$ is introduced in \cref{remark:subgroup},
and the action of~$G$ on~$A_0/K_0$ is discussed in \cref{lemma:action}.
The subgroup~$H\subseteq G$ is introduced in \cref{lemma:subgroup}.
With this notation, we have the following.

\begin{theorem}
  \label{theorem:albanese-morphism}
  The Albanese morphism~$\alb_X\colon X\to\Alb(X)$ is
  the morphism in~$f\colon X\to B$ from \eqref{equation:albanese-morphism-to-be}.
  This means that
  \begin{enumerate}
    \item $\Alb(X)\cong(A_0/K_0)/G$;
    \item for every~$a_0\in A_0$
      we have that the Albanese fiber is~$\alb_X^{-1}([a_0+K_0]_G)\cong(\{a_0\}+A_1)/H$.
  \end{enumerate}
\end{theorem}
Thus, the Albanese fibers are either abelian varieties or hyperelliptic varieties,
and we can precisely describe them.

\begin{proof}
  It suffices to show that the isogeny $\bar{f} \colon \Alb(X) \to B$
  in \Cref{lemma:factorization} is an isomorphism,
  for which it suffices to show that it has connected fibers.
  But by \Cref{proposition:fibers}, the fibers of $f$ are connected,
  so the fibers of $\bar{f}$ must be connected as well.
\end{proof}
The first part of the statement in \Cref{theorem:albanese-morphism}
is a generalization of \cite[Proposition 25]{MR3404712},
which assumes that $G$ is cyclic.

\begin{remark}
  \label{remark:existing}
  The statement of \cref{theorem:albanese-morphism}
  is related to Ueno's proof of the characterization of Kummer manifolds of Kodaira dimension~0
  (not to be confused with hyperk\"ahler varieties of~$\mathrm{Kum}^n$-type,
  which are a very special case),
  as can be seen by inspecting the second half of the proof of \cite[Theorem~7.13]{MR360582}.
  In op.~cit.~a Kummer manifold is a nonsingular model of the quotient of an abelian variety~$A$
  by a finite group~$G$ (without conditions on the action),
  and the characterization is given in terms of a birational model
  being an \'etale fiber bundle whose fiber
  is again a Kummer manifold of Kodaira dimension~0.

  Hyperelliptic varieties are a special and particularly explicit case of these,
  and \cref{theorem:albanese-morphism} upgrades the results
  from \cite[\S7]{MR360582}
  to more explicit and fully effective results in this special case.
\end{remark}

\begin{remark}
  \label{remark:torsion-order-of-fiber}
  Let $F$ be a (general) fiber of the Albanese morphism of $X$.
  Then $\mathrm{K}_{F}$ is also torsion, and its order divides the order of $\mathrm{K}_{X}$.
  Indeed, the normal bundle of $F$ in $X$ is trivial, so by adjunction we have $\mathrm{K}_{X}|_{F} = \mathrm{K}_{F}$.
  Therefore, if $m\mathrm{K}_{X} \sim 0$ for some integer $m > 0$, then $m\mathrm{K}_{F} \sim 0$ as well.
  The explicit description of $F$ in \Cref{theorem:albanese-morphism} confirms this, because $H$ is a subgroup of $G$.
\end{remark}

\section{Fibers of the Albanese morphism}
\label{section:albanese-fibers}

In this section we take a closer look at the possible fibers of the Albanese morphism.
First we study the case of cyclic hyperelliptic varieties,
where the Albanese fiber is again cyclic,
and we give a useful sufficient condition for the Albanese fibers to be abelian varieties.
To illustrate these results we discuss the well-known case of bielliptic surfaces.
Conversely,
for a given abelian variety we give a construction of a (cyclic) hyperelliptic variety
with that abelian variety as Albanese fiber.
Finally, we study the Albanese fibers of hyperelliptic threefolds,
observing the interesting phenomenon that not all bielliptic surfaces appear as fiber.

\paragraph{Albanese fibers of cyclic hyperelliptic varieties}
We start by making the following observation:
\begin{lemma}
  \label{lemma:cyclic-fibers}
  Let $X = A/G$ be a cyclic hyperelliptic variety.
  Then its Albanese fibers are either abelian varieties
  or lower-dimensional cyclic hyperelliptic varieties.
\end{lemma}

\begin{proof}
  The explicit description in \Cref{theorem:albanese-morphism}
  shows that the fibers are either abelian or cyclic hyperelliptic,
  because $H$ is a subgroup of the cyclic group $G$,
  hence cyclic itself.
  The dimension bound follows from \Cref{remark:cyclic-irregularity-bound}.
\end{proof}

The following result gives some sufficient condition
for the Albanese fibers of a cyclic hyperelliptic variety
to be abelian varieties.

\begin{proposition}
  \label{proposition:cyclic-implies-abelian-fibers}
  Let $X = A/G$ be a cyclic hyperelliptic variety and $g$ a generator of $G$.
  Suppose that, besides the eigenvalue $1$,
  all eigenvalues of $\tilde{\rho}_X(g)$ are roots of unity of the same order.
  Then the Albanese fibers of $X$ are isomorphic to the abelian variety~$A_1$
  from \cref{lemma:linear-action}.
\end{proposition}
The condition on the eigenvalues holds whenever~$G\cong\mathbb{Z}/p\mathbb{Z}$
with~$p$ prime.

\begin{proof}
  As described at the beginning of \cref{section:albanese-morphism}, we may write
  \begin{equation}
    A = (A_0 \times A_1)/K,
  \end{equation}
  where a generator $g$ of $G$ acts on $A_0$ by translation and $\rho_X(g)|_{A_1}$
  does not have the eigenvalue $1$.
  As explained in~\cite[Lemma 3.5]{MR1862215},
  we can write the action of $g$ on $A = (A_0 \times A_1)/K$ as
  \begin{equation}
    g((a_0,a_1) + K) = (a_0 + t_0, ~ \alpha a_1) + K,
  \end{equation}
  where $\alpha \colon A_1 \to A_1$ is a Lie group automorphism of $A_{1}$ of order dividing $\#G$ and $t_{0} \in A_{0}$.
  Our assumption on the eigenvalues then means that $\alpha^m$ does not have the eigenvalue $1$ for every $1 \leq m < \#G$.
  If $H$ denotes the subgroup from \cref{lemma:subgroup}, the freeness of the action of $G = \langle g \rangle$
  implies that $mt_0 \notin K_{0}$ for every $1 \leq m < \#G$, i.e., $H$ is trivial.
  Indeed, suppose on the contrary that $mt_{0} \in K_{0}$,
  so that we can find some $k_{1} \in K_{1}$ with $(mt_{0},k_{1}) \in K$.
  We have
  \begin{equation}
    g^{m}((a_{0},a_{1})+K) = (a_{0} + mt_{0}, \alpha^{m} a_{1}) + K,
  \end{equation}
  and the assumption on the eigenvalues implies that $\alpha^{m}-\operatorname{id}$ is invertible.
  Therefore, we can find some $a_{1} \in A_{1}$ such that $\alpha^{m}(a_{1}) - a_{1} = k_{1}$.
  Hence we have
  \begin{equation}
    g^{m}((0,a_{1})+K) = (mt_{0}, \alpha^{m} a_{1}) + K = (0, \alpha^{m} a_{1} - k_{1}) + K = (0, a_{1}) + K,
  \end{equation}
  so that $(0,a_{1}) + K$ is fixed by $g^{m}$, a contradiction.
  Therefore, $H = 0$.
  By \cref{proposition:fibers}, we obtain that the Albanese fibers of $X$ are translates of $A_1$.
\end{proof}

\begin{example}
  \label{example:bielliptic-fibers}
  The assumption on the eigenvalues in \cref{proposition:cyclic-implies-abelian-fibers} is necessary
  for the conclusion,
  as the following example shows.
  Let
  \begin{equation}
    A \colonequals (E_{\tau_0} \times E_{\tau_1} \times E_{i})/K, \qquad K \colonequals \left\langle \left( \frac12, \frac12, 0 \right)\right\rangle, \qquad E_{\tau_j} \colonequals \mathbb{C}/(\mathbb{Z} + \tau_j \mathbb{Z}) \text{ for } \tau_j \in \mathbb{H}, \qquad E_i \colonequals \mathbb{C}/(\mathbb{Z} + i \mathbb{Z}),
  \end{equation}
  where $\mathbb{H}$ denotes the upper half-plane.
  Define a biholomorphic automorphism $\tilde{g}$ of $E_{\tau_0} \times E_{\tau_1} \times E_{i}$ via
  \begin{equation}
    \tilde{g}(z_0,z_1,z_2) \colonequals \left(z_0 + \frac14, -z_1, iz_2\right).
  \end{equation}
  The linear part of $\tilde{g}$ fixes $K$,
  so that $\tilde{g}$ descends to a biholomorphic automorphism $g$ of $A$.
  We claim that~$A/G$ is a hyperelliptic threefold
  with holonomy group $G \colonequals \langle g  \rangle \cong \mathbb{Z}/4\mathbb{Z}$.
  Since
  \begin{equation}
    g^j((z_0,z_1,z_2) + K) = \left(z_0 + \frac{j}{4}, (-1)^j z_1, i^j z_2\right) + K,
  \end{equation}
  the group $G = \langle g \rangle \subseteq \Bihol(A)$
  has order $4$ and contains no translations.
  It is clear from the definitions that $g$ and $g^3$ act freely on $A$.
  Furthermore, $g^2$ acts freely as well, since
  \begin{equation}
    g^2((z_0,z_1,z_2) + K) = (z_0,z_1,z_2) + K \iff \left(\frac12, 0, -2z_2\right) \in K,
  \end{equation}
  which is never satisfied, regardless of the choice of $(z_0,z_1,z_2)$.
  This completes the proof that $X \colonequals A/G$
  is hyperelliptic with cyclic holonomy group of order $4$.

  We shall now calculate the Albanese fiber of $X$
  and show that it is a bielliptic surface.
  Using the notation of \cref{section:albanese-morphism}, we have
  \begin{equation}
    A_0 = E_{\tau_0}, \qquad A_1 = E_{\tau_1} \times E_i, \qquad K_0 = \left\langle \frac12 \right\rangle \subseteq A_0, \qquad K_1 =  \left\langle \left(\frac12, 0\right) \right\rangle \subseteq A_1,
  \end{equation}
  so that the subgroup $H \subseteq G$ introduced in \cref{lemma:subgroup} is generated by $g^2$.
  To obtain an explicit description of the Albanese fiber $A_1/H$ of $X$, recall that
  \begin{equation}
    g^2((z_0,z_1,z_2) + K) = \left(z_0, z_1 + \frac12, -z_2\right) + K,
  \end{equation}
  so that $H = \langle g^2 \rangle$ acts on $A_1$ via $g^2(z_1,z_2) = \left(z_1 + \frac12, -z_2\right)$.
  Thus $A_1/H$ is a bielliptic surface with holonomy group $\mathbb{Z}/2\mathbb{Z}$.
\end{example}

Conversely, the following gives a construction of an hyperelliptic variety
with a given abelian variety as Albanese fiber.

\begin{proposition}
  \label{proposition:abelian-fibers}
  Let $A_1$ be an abelian variety of dimension $m$.
  Then, for every $n \geq 1$, there exists a hyperelliptic variety of dimension $n + m$
  whose Albanese fibers are isomorphic to $A_1$.
\end{proposition}

\begin{proof}
 Let $A_0$ be an abelian variety of dimension $n$
 and denote by $\tau \in A_0[2]$ a non-trivial $2$-torsion point of $A_0$.
 According to \cref{proposition:cyclic-implies-abelian-fibers},
 the Albanese fiber of the $(n+m)$-dimensional hyperelliptic variety
 \begin{equation}
   (A_{0} \times A_{1})/G, \qquad G = \langle (z_0,z_1) \mapsto (z_0 + \tau, -z_1)\rangle \cong \mathbb{Z}/2\mathbb{Z}
 \end{equation}
 is $A_1$.
\end{proof}

\paragraph{Albanese fibers of bielliptic surfaces}

Bielliptic surfaces are $2$-dimensional hyperelliptic varieties.
They are fully classified:~there are exactly seven families of bielliptic surfaces,
see \cite[Example 3.2.4]{2211.07998} and the references therein.
All of them are cyclic hyperelliptic varieties to which
\Cref{proposition:cyclic-implies-abelian-fibers} applies,
and their canonical divisors have orders~2, 3, 4, or~6.

\Cref{table:bielliptic-surfaces} below illustrates \cref{theorem:albanese-morphism}
in the case of bielliptic surfaces,
showing the seven families, their Albanese and
their Albanese fibers.
Given the data in the table,
we obtain a bielliptic surface via the action of $G$ on $(A_0 \times A_1)/K$.
Since $G$ is cyclic in all seven cases, we only give the order and the action of a generator of $G$.

\begin{table}[ht]
  \centering
  \begin{tabular}{ccccccc}
    \toprule
    $A_0$    & $A_1$         & $K$                                                    & $\#G$ & action of generator~$g$ of~$G$               & $\Alb(X)$                                           & Albanese fiber \\
    \midrule
    $E_\tau$ & $E_{\tau'}$   & $\{0\}$                                                & $2$   & $g(z_1,z_2) = (z_1 + \frac12, -z_2)$         & $E_\tau/\langle \frac12 \rangle$                    & $E_{\tau'}$ \\\addlinespace
    $E_\tau$ & $E_{\tau'}$   & $\langle(\frac{\tau}{2}, \frac12)\rangle $             & $2$   & $g(z_1,z_2) = (z_1 + \frac12, -z_2)$         & $E_{\tau}/\langle  \frac12, \frac{\tau}{2} \rangle$ & $E_{\tau'}$ \\\addlinespace
    $E_\tau$ & $E_{\zeta_3}$ & $\{0\}$                                                & $3$   & $g(z_1,z_2) = (z_1 + \frac13, \zeta_3 z_2)$  & $E_{\tau}/\langle \frac13 \rangle$                  & $E_{\zeta_3}$ \\\addlinespace
    $E_\tau$ & $E_{\zeta_3}$ & $\langle (\frac{\tau}{3}, \frac{1-\zeta_3}{3})\rangle$ & $3$   & $g(z_1,z_2) = (z_1 + \frac13, \zeta_3 z_2)$  & $E_{\tau}/\langle \frac13, \frac{\tau}{3} \rangle$  & $E_{\zeta_3}$ \\\addlinespace
    $E_\tau$ & $E_i$         & $\{0\}$                                                & $4$   & $g(z_1,z_2) = (z_1 + \frac14, i z_2)$        & $E_{\tau}/\langle \frac14 \rangle$                  & $E_i$ \\\addlinespace
    $E_\tau$ & $E_i$         & $\langle (\frac{\tau}{2}, \frac{1+i}{2})\rangle$       & $4$   & $g(z_1,z_2) = (z_1 + \frac14, i z_2)$        & $E_{\tau}/\langle \frac14, \frac{\tau}{2}\rangle$   & $E_i$ \\\addlinespace
    $E_\tau$ & $E_{\zeta_3}$ & $\{0\}$                                                & $6$   & $g(z_1,z_2) = (z_1 +  \frac16, \zeta_6 z_2)$ & $E_{\tau}/\langle \frac16 \rangle$                  & $E_{\zeta_3}$ \\
    \bottomrule
  \end{tabular}
  \caption{Albanese fibers of bielliptic surfaces}
  \label{table:bielliptic-surfaces}
\end{table}

\paragraph{Albanese fibers of hyperelliptic threefolds}
It follows from \Cref{proposition:abelian-fibers} that
every elliptic curve or abelian surface arises as the Albanese fiber of some
hyperelliptic threefold.
The only other candidate for the Albanese fibers of irregular hyperelliptic threefolds are bielliptic surfaces,
for hyperelliptic threefolds of irregularity~$q_X=1$.
We have already seen in \cref{example:bielliptic-fibers} that
some bielliptic surfaces occur as Albanese fibers
of hyperelliptic threefolds.
The following example gives another construction of (non-cyclic) hyperelliptic threefolds
with bielliptic surfaces as Albanese fiber.
However, we will show in \cref{proposition:not-all-bielliptic-occur}
that not every bielliptic surface occurs,
thus adding an interesting twist to the geography of hyperelliptic varieties.

\begin{example}
  \label{example:hyperelliptic-fibers}
  Let $m \in \{2,3\}$.
  Let $E_1, E_2$ be elliptic curves that admit a linear automorphism of order $m$.
  More precisely, denoting by $\mathbb{H}$ the upper half-plane:
  \begin{itemize}
    \item If $m = 2$, then $E_1, E_2$ can be chosen to be $\mathbb{C}/(\mathbb{Z} + \tau_{i} \mathbb{Z})$ for arbitrary $\tau_1, \tau_2 \in \mathbb{H}$ respectively,
    because every elliptic curve has $-\mathrm{id}$ as an automorphism.
    \item If $m = 3$, then $E_1, E_2$ can be chosen to be $\mathbb{C}/(\mathbb{Z} + \zeta_3 \mathbb{Z})$, where $\zeta_3 = \exp\left(\frac{2 \pi i}{3}\right)$.
  \end{itemize}
  Denote by $\zeta_m$ a primitive $m$th root of unity so that $E_1$ and $E_2$ admit the automorphism $z \mapsto \zeta_m z$.
  Furthermore, let $E_0 = \mathbb{C}/(\mathbb{Z} + \tau_0 \mathbb{Z})$ for an arbitrary $\tau_0 \in \mathbb{H}$.

  Let $\tilde{g}_{1}, \tilde{g}_{2} \in \Aut(E_0 \times E_1 \times E_2)$ be the automorphisms defined by
  \begin{equation}
    \begin{aligned}
      \tilde{g}_1(z_0, z_1,z_2) &\colonequals \left(z_0 + ~ \frac{1}{m}, ~ \zeta_m z_1, ~ z_2 \right), \\
      \tilde{g}_2(z_0, z_1,z_2) &\colonequals \left(z_0 + ~ \frac{\tau_3}{m}, ~ z_1, ~ \zeta_m z_2 \right).
    \end{aligned}
  \end{equation}
  Consider the induced action of $\tilde{G} \colonequals \langle \tilde{g}_{1}, \tilde{g}_{2} \rangle \cong (\mathbb{Z}/m\mathbb{Z})^{2}$ on $E_0 \times E_1 \times E_2$.
  Moreover, let
  \begin{equation}
    t \colonequals
    \begin{cases}
      1/2, & \text{if } m = 2, \\
      (1-\zeta_3)/3, & \text{if } m = 3.
    \end{cases}
  \end{equation}
  Then $t \neq 0$ is a fixed point of $z \mapsto \zeta_m z$ on $E_1$ and $E_2$.

  Consider the subgroup $K \colonequals \langle (\frac1m, 0,t)\rangle$ of $E_0 \times E_1 \times E_2$.
  The linear parts of $\tilde{g}_1$ and $\tilde{g}_2$ map $K$ to $K$, respectively, so that $\tilde{g}_1$ and $\tilde{g}_2$
  descend to automorphisms $g_{1}$ and $g_{2}$ of $A \colonequals (E_0 \times E_1 \times E_2)/K$ given by the same formulae.

  One checks that $G \colonequals \langle g_1, g_2\rangle \subseteq \Aut(A)$, is still isomorphic to $(\mathbb{Z}/m\mathbb{Z})^2$ and acts freely on $A$.
  Hence $X \colonequals A/G$ is a hyperelliptic threefold.
  Using the notation of \cref{section:albanese-morphism}, we have that
  \[
  A_0 = E_0, \qquad A_1 = E_1 \times E_2, \qquad K_0 = \left\langle \left(\frac1m\right)\right\rangle \subseteq A_0, \qquad K_1 = \langle (0,t) \rangle \subseteq A_1.
  \]
  Thus the subgroup $H \subseteq G$ defined in \cref{lemma:subgroup} is generated by $g_1$.
  By our description of the Albanese fibers of $X$ given in \cref{proposition:fibers},
  we see that these are given by
  \begin{equation}
    (E_1 \times E_2)/H, \qquad \text{where} \qquad H = \langle g_1 \rangle.
  \end{equation}
  Here, $g_1$ acts on $E_1 \times E_2$ by $(z_1, z_2) \mapsto (\zeta_m z_1, ~ z_2 + t)$, so that the Albanese fibers are bielliptic surfaces with holonomy group $\mathbb{Z}/m\mathbb{Z}$.
\end{example}

\begin{remark}
  A natural question is if the construction given in \cref{example:hyperelliptic-fibers} will work for $m \in \{4,6\}$, too.
  The answer is no:
  \begin{itemize}
    \item For $m = 4$, the only choice for $t$ is $t = (1+i)/2$, which has order $2$ (see for instance \cite[Proposition~2.3]{MR4809689}).
    Thus the analogous action of $G$ on $A$ will not be free.
    \item For $m = 6$, there is no non-zero fixed point of $z \mapsto -\zeta_3 z$ (see loc. cit.).
  \end{itemize}
\end{remark}

In fact, the following no-go result holds.

\begin{proposition}
  \label{proposition:not-all-bielliptic-occur}
  Not all bielliptic surfaces occur as Albanese fibers of hyperelliptic threefolds.
\end{proposition}

\begin{proof}
  Consider the $1$-dimensional family of bielliptic surfaces $S_\tau = T_\tau/G'$,
  whose holonomy group is~$G' \cong \mathbb{Z}/3\mathbb{Z}$,
  defined by
  \begin{equation}
    T_\tau = E_\tau \times E_{\zeta_3}, \qquad \text{where} \quad E_\tau = \mathbb{C}/(\mathbb{Z} + \tau \mathbb{Z}) \text{ for } \tau \in \mathbb{H},\quad E_{\zeta_3} = \mathbb{C}/(\mathbb{Z} + \zeta_3 \mathbb{Z}),
  \end{equation}
  where $\mathbb{H}$ denotes the upper half-plane.
  We will show that $S_\tau$ does not occur as the Albanese fiber of any hyperelliptic threefold for general $\tau$.

  Let $X = A/G$ be a hyperelliptic threefold.
  According to the work of Uchida--Yoshihara \cite{MR400271}, the group $G$ falls into one of three categories:
  \begin{enumerate}[label=(\roman*)]
    \item\label{item:threefold-group-1} $G$ is cyclic of order $d \in \{ 2, 3, 4, 5, 6, 8, 10, 12\}$,
    \item\label{item:threefold-group-2} $G$ is an abelian group isomorphic to a group of the form $\mathbb{Z}/d_1 \mathbb{Z} \times \mathbb{Z}/d_2 \mathbb{Z}$ for some
    \begin{equation}
      (d_1, d_2) \in \{ (2,2), (2,4), (2,6), (2,12), (3,3), (3,6), (4,4), (6,6) \} \text{, or}
    \end{equation}
    \item\label{item:threefold-group-3} $G \cong \mathrm{D}_{4}$ is isomorphic to the dihedral group of order $8$.
  \end{enumerate}
  We will investigate these cases separately.
  By \cref{proposition:cyclic-implies-abelian-fibers}, the Albanese fibers of $X$ are elliptic curves or abelian surfaces
  if $X$ is of type \ref{item:threefold-group-1} and~$d$ is a prime.
  Furthermore, since the holonomy group of $S_\tau$ is cyclic of order $3$, one has $3 ~ | ~ d$, if $S_\tau$ is the Albanese fiber of $X$ of type \ref{item:threefold-group-1}.
  Thus we are left with the cases $d \in \{6, 12\}$.
  Using the notation from \cref{section:albanese-morphism}, wee see that $d = 12$ is easily excluded,
  since~$A_1$ is isogenous to a product of CM elliptic curves in this case, see \cite[Table 2, p.~509]{MR1862215}.
  However, for general $\tau \in \mathbb{H}$, the torus $T_\tau$ is not a product of CM elliptic curves.
  Hence a general~$S_\tau$ cannot occur as an Albanese fiber of $X$ if $d = 12$.
  The remaining case, $d = 6$, is divided into several subcases in the cited table.
  There are only two of these we have to consider, since~$A_1$ is again isogenous to a product of CM elliptic curves in all other cases.
  In our remaining cases, denoting by $g$ a generator of $G \cong \mathbb{Z}/6\mathbb{Z}$, we may assume that
  \begin{equation}
    \tilde{\rho}_X(g) = \operatorname{diag}(1, -1, \pm \zeta_3).
  \end{equation}
  Here,  $A_0 = E_{\tau'}$ ($\tau' \in \mathbb{H}$) is an elliptic curve, $A$ is isogenous to~$A_0 \times A_1$, and $A_1$ is isogenous to $E_\tau \times E_{\zeta_3}$ for some $\tau \in \mathbb{H}$. We may therefore write
  \begin{equation}
    A = (E_{\tau'} \times E_\tau \times E_{\zeta_3})/K,
  \end{equation}
   and after changing the origin of $ E_\tau \times E_{\zeta_3}$, we may assume that $g$ acts on $A$ via
  \begin{equation}
    g((z_1,z_2,z_3)+K) = (z_1 + t, -z_2, \pm \zeta_3 z_3)+K.
  \end{equation}
  Since $g^m$ acts freely for every $1 \leq m < 6$, the order of $t$ in~$E_1$ is equal to $6$.
  If~$S_\tau$ occurred as the Albanese fiber of $X$, then the subgroup $H \subseteq G$ of \cref{lemma:subgroup} necessarily has order $3$.
  In this case, $H$ is generated by
  \begin{equation}
    g^2((z_1,z_2,z_3)+K) = (z_1 + 2t, z_2, \zeta_3^2 z_3)+K.
  \end{equation}
  By definition of $H$, there exists an element $(2t, x_2, x_3) \in K$, so that the above action of $g^2$ can also be written as
  \begin{equation}
    g^2((z_1,z_2,z_3)+K) = (z_1, z_2 - x_2, \zeta_3^2 z_3 - x_3)+K.
  \end{equation}
  Abusing notation slightly, we view $g$ as an automorphism of $E_{\tau'} \times E_\tau \times E_{\zeta_3}$ whose linear part $\tilde{\rho}_X(g)$ maps $K$ to $K$.
  Then
  \begin{equation}
    (\operatorname{id} - \tilde{\rho}_X(g))(2t, x_2, x_3) = (0, 2x_2, (1\mp\zeta_3)x_3) \in K.
  \end{equation}
  It follows that
  \begin{equation}
    g^4((z_1,z_2,z_3)+K) = (z_1, z_2 - 2x_2, \zeta_3 z_3 - \zeta_3^2 x_3 - x_3)+K = (z_1, z_2, \zeta_3 z_3 - \zeta_3^2 x_3 \mp \zeta_3 x_3)+K
  \end{equation}
  has a fixed point on $A$, a contradiction.
  Thus $S_\tau$ cannot occur as the Albanese fiber of threefolds of type~\ref{item:threefold-group-1}.

  For hyperelliptic varieties of type \ref{item:threefold-group-3},
  it was shown in \cite[Theorem 0.1]{MR4186481} that the complex representation
  of hyperelliptic threefolds of type \ref{item:threefold-group-3} does not contain the trivial representation,
  thus they have irregularity $0$ (see \cref{remark:irregularity}).
  We are thus left with type \ref{item:threefold-group-2}.

  Let $X = A/G$ be a hyperelliptic threefold of type \ref{item:threefold-group-2}.
  Using the notation from \cref{section:albanese-morphism},
  we know that the Albanese fibers of $X$ are quotients of $A_1$ by a subgroup~$H \subseteq G$.
  Thus, in order for $S_\tau$ to occur as the Albanese fiber of $X$, we must have $3 \mid d_2$,
  and hence
  \begin{equation}
    (d_1,d_2) \in \{(2,6), (2, 12), (3,3), (3,6), (6,6)\}.
  \end{equation}
  According to \cite[Table 3, p.~509]{MR1862215}, the abelian variety $A_1$ is isogenous to a product of elliptic curves.
  Inspecting the table, one sees that if $(d_1,d_2) \neq (2,6)$, then $A_1$ is isogenous to a product of CM elliptic curves and we conclude as in case \ref{item:threefold-group-1} above.

  Consequently, it suffices to show that no $S_{\tau}$ can occur as an Albanese fiber of $X$ if $(d_1, d_2) = (2,6)$.
  Inspecting the cited table again, we see that the case $(d_1, d_2) = (2,6)$ is divided into four subcases.
  Only the first of these has to excluded, since $A_1$ is again isogenous to a product of CM elliptic curves in the other three subcases.
  As above, we may write
  \begin{equation}
    A = (E_{\tau'} \times E_\tau \times E_{\zeta_3})/K.
  \end{equation}
  Denote by $g_1$, $g_2$ generators of $G \cong \mathbb{Z}/2\mathbb{Z} \times \mathbb{Z}/6\mathbb{Z}$ of respective orders $2$ and $6$.
  According to the cited table, we may assume that
  \begin{equation}
    \tilde{\rho}_X(g_1) = \operatorname{diag}(1,-1,1), \qquad \tilde{\rho}_X(g_2) = \operatorname{diag}(1,1,\zeta_6).
  \end{equation}
  Thus, the action of $G$ on $A$ has the form
  \begin{equation}
    g_1((z_1, z_2, z_3)+K) = \left(z_1 + s', -z_2, z_3 + s\right)+K, \qquad
    g_2((z_1, z_2, z_3)+K) = \left(z_1 + t', z_2 + t, \zeta_6 z_3 \right)+K,
  \end{equation}
  where $s',t' \in E_{\tau'}$, $s \in E_{\zeta_3}$, and $t \in E_\tau$.

  We will now use the important property that each of the three elliptic curves $E_{\tau'}$, $E_\tau$, and $E_{\zeta_3}$ are abelian subvarieties of $A$,
  i.e., the composed maps
  $E_{\tau'} \hookrightarrow E_{\tau'} \times E_\tau \times E_{\zeta_3} \to A$ (and similarly for $E_\tau$ and $E_{\zeta_3}$) are injective.
  This implies that
  \begin{equation}
    \label{equation:unit-vector}
    \text{if } (x_1,x_2,x_3) \in K \text{ with } x_{i} = x_{j} = 0 \text{ for some } i \neq j, \text{ then } (x_1,x_2,x_3) = 0.
  \end{equation}
  Abusing notation slightly again, we may view $g_1$ and $g_2$ as elements of $\Bihol(E_{\tau'} \times E_\tau \times E_{\zeta_3})$ such that~$\rho_X(g_1)$ and~$\rho_X(g_2)$ map $K$ to $K$, respectively.
  We then obtain for every $(x_1,x_2,x_3) \in K$, that
  \begin{equation}
    \begin{aligned}
      (\operatorname{id} - \rho_X(g_1))(x_1,x_2,x_3) &= (0, 2x_2, 0) \in K, \\
      (\operatorname{id} - \rho_X(g_2))(x_1,x_2,x_3) &= (0, 0, (1 - \zeta_6)x_3) \in K.
    \end{aligned}
  \end{equation}
  By (\ref{equation:unit-vector}), $x_2 \in E_\tau[2]$ and $(1-\zeta_6) x_3 = 0$ in $E_{\zeta_3}$.
  By \cite[Proposition 2.3]{MR4809689}, this implies that $x_3 = 0$.
  Hence,
  \begin{equation}
    2 \cdot (x_1,x_2,x_3) = (2x_1, 2x_2, 2x_3) = (2x_1,0,0) \in K
  \end{equation}
  implies that $2x_1 = 0$ in $E_{\tau'}$, i.e., we have shown that
  \begin{equation}
    K \subseteq E_{\tau'}[2] \times E_\tau[2] \times \{0\}.
  \end{equation}
  In particular, the image $K_0$ of the projection $K \to E_{\tau'}$ is contained in $E_{\tau'}[2]$.
  Now,
  \begin{equation}
    g_1^2((z_1,z_2,z_3)+K) = (z_1 + 2s', z_2, z_3 + 2s)+K.
  \end{equation}
  Viewing $g_1^2$ again as an automorphism of $E_{\tau'} \times E_\tau \times E_{\zeta_3}$, we obtain that $(2s', 0, 2s) \in K$,
  i.e., $2s = 0$ and hence $2s' = 0$ by (\ref{equation:unit-vector}).
  Furthermore,
  \begin{equation}
    \begin{aligned}
      (g_1 \circ g_2)((z_1,z_2,z_3)+K) &= \left(z_1 + s' + t', -z_2 - t, \zeta_6 z_3 + s\right)+K, \\
      (g_2 \circ g_1)((z_1,z_2,z_3)+K) &= \left(z_1 + s' + t', -z_2 + t, \zeta_6 z_3 + \zeta_6 s\right)+K.
    \end{aligned}
  \end{equation}
  Since $g_1 \circ g_2 = g_2 \circ g_1$, we obtain that $(0, 2t, (\zeta_6 - 1)s) \in K$.
  Hence $s = 0$ and thus $2t = 0$, too.
  As above, the condition that $g_2$ has order $6$ implies that $(6t', 2t, 0) = (6t',  0,   0) \in K$,
  i.e., $6t' = 0$.
  Clearly, $2t' \neq 0$, because $g_2^2((z_1,z_2,z_3)+K) = (z_1 + 2t', z_2, \zeta_3 z_3)+K$ and $g_{2}^{2}$ acts freely.

  This means that the subgroup $H$ of $G$ containing all elements of $G$
  which act on $E_{\tau'}$ by a translation in~$K_0$ consists entirely
  of elements of order at most $2$ (in fact, $H \subseteq \langle g_2^3 \rangle$).
  It follows that the Albanese fibers~$A_1/H$ of $X$ do not have $\mathbb{Z}/3\mathbb{Z}$-holonomy.
  Thus no $S_\tau$ can occur as the Albanese fiber of $X$ in the last remaining case $(d_1,d_2) = (2,6)$.
\end{proof}

\section{Indecomposability of the derived category}
\label{section:indecomposability}
As an illustration of the applicability of \cref{theorem:fiber-albanese}
we will consider the derived category~$\derived^\bounded(X)$ of a hyperelliptic variety.
For a survey of the role of semiorthogonal decompositions of derived categories
in algebraic geometry
one is referred to Kuznetsov's ICM surveys \cite{MR4680279,MR3728631}.
As explained in the introduction,
we conjecture that for hyperelliptic varieties
there are \emph{no} such decompositions.

We will discuss two approaches to verify this in many cases.

\paragraph{Stable indecomposability}
The following notion is (the version over~$\mathbb{C}$ of) \cite[Definition~1.3]{MR4582884}.
\begin{definition}
  \label{definition:stably-indecomposable}
  Let~$X$ be a scheme over~$\mathbb{C}$.
  We say it is \emph{noncommutatively stably semiorthogonally indecomposable}
  (or \emph{stably indecomposable} for short)
  if
  \begin{itemize}
    \item for every proper~$\Perf(X)$\dash linear category~$\mathcal{D}$
      which admits a classical generator,
    \item and every left admissible subcategory~$\mathcal{A}\subseteq\mathcal{D}$,
  \end{itemize}
  we have that~$\mathcal{A}$ is closed under the action of~$\Perf(X)$.
\end{definition}
Remark that this is not a property of~$\Perf(X)$,
rather it is a property of the tensor triangulated category~$(\Perf(X),-\otimes^{\mathbf{L}}-)$,
which recovers~$X$ itself by Balmer reconstruction.
In all our examples, $X$ will be smooth and projective,
so that~$\Perf(X)=\derived^\bounded(X)$.

By \cite[Lemma~2.13]{MR4582884} we have that~$X$ being stably indecomposable
implies that~$\derived^\bounded(X)$ is indecomposable.
It is however quite a bit stronger:
if~$Z\subset X$ is a smooth connected closed subvariety
of a stably indecomposable variety~$X$,
then~$\derived^\bounded(Z)$ is indecomposable.
Thus, while a K3~surface has an indecomposable derived category,
if there is a smooth rational curve on it,
it cannot be stably indecomposable \cite[Example~2.16]{MR4582884}.

The following is \cite[Theorem~1.5]{MR4582884},
and it is the main result that we will use for hyperelliptic varieties.
\begin{theorem}[Pirozhkov]
  \label{theorem:pirozhkov}
  Let~$f\colon X\to Y$ be a flat morphism between smooth and proper varieties.
  If
  \begin{itemize}
    \item~$Y$ is stably indecomposable, and
    \item for all closed points~$y\in Y$ the fiber~$X_y$ is stably indecomposable,
  \end{itemize}
  then~$X$ is stably indecomposable.
\end{theorem}

Abelian varieties are stably indecomposable \cite[Proposition 3.4]{MR4582884}.
From this we deduce the following two corollaries.

\begin{corollary}
  \label{corollary:cyclic}
  Let $X$ be a cyclic hyperelliptic variety.
  Then~$X$ is stably indecomposable.
\end{corollary}

\begin{proof}
  By \Cref{lemma:cyclic-fibers},
  the fibers of the Albanese morphism of $X$ are either abelian varieties
  or lower-dimensional cyclic hyperelliptic varieties.
  Therefore, applying \Cref{theorem:pirozhkov} inductively on the dimension of $X$,
  we deduce the result.
\end{proof}

\begin{example}
  \label{example:bielliptic-surface}
  From \Cref{corollary:cyclic}
  we recover the fact that bielliptic surfaces have (stably) indecomposable derived categories.
  This was already obtained in \cite[Corollary~4.2]{MR4582884}.
\end{example}

A different application of~\cref{theorem:pirozhkov} yields the following.
\begin{corollary}
  \label{corollary:big-irregularity}
  Let~$X$ be a hyperelliptic variety of irregularity~$q_X=\dim X-2$ or~$\dim X-1$.
  Then~$X$ is stably indecomposable.
\end{corollary}

\begin{proof}
  By \cref{theorem:albanese-morphism} we know that under this assumption
  the fibers are
  either elliptic curves,
  abelian surfaces,
  or bielliptic surfaces.
  These are stably indecomposable by \cite[Proposition 3.4]{MR4582884}
  and \cref{example:bielliptic-surface},
  so an application of \cref{theorem:pirozhkov} yields the result.
\end{proof}

The only threefolds not covered by \cref{corollary:big-irregularity}
have~$q_X=0$,
i.e., they are regular.
The following easy observation deals with the indecomposability of their derived categories.
\begin{lemma}
  \label{lemma:regular-hyperelliptic-threefold}
  Let~$X$ be a regular hyperelliptic threefold.
  Then~$\omega_X\cong\mathcal{O}_X$,
  and thus~$\derived^\bounded(X)$ is indecomposable.
\end{lemma}

\begin{proof}
  The variety~$X$ is a finite \'etale quotient of an abelian variety $A$, so that
  \begin{equation}
    0 = \chi(A,\mathcal{O}_{A}) = d\chi(X,\mathcal{O}_{X})
  \end{equation}
  for some $d > 0$.
  In particular,
  \begin{equation}
    \chi(X,\mathcal{O}_{X}) = 1 - 0 + \hh^2(X,\mathcal{O}_{X}) - \hh^3(X,\mathcal{O}_{X}) = 0.
  \end{equation}
  Therefore, $\hh^3(X,\mathcal{O}_{X}) = 1$.
  This implies that $\omega_{X}$ admits a non-zero global section.
  But since $X$ is hyperelliptic, $\omega_{X}$ is a torsion line bundle.
  Hence $\omega_{X} \cong \mathcal{O}_{X}$ and $\derived^\bounded(X)$ is indecomposable by Serre duality.
\end{proof}

\begin{example}
  \label{example:D4-threefold}
  An example of a hyperelliptic threefold
  satisfying the condition of \cref{lemma:regular-hyperelliptic-threefold}
  is given in \cite{MR4186481}.
  Its Hodge diamond is
  \begin{equation}
    \begin{array}{ccccccc}
        &   &   & 1 \\
        &   & 0 &      & 0 \\
        & 0 &   & 2    &      & 0 \\
      1 &   & 2 &      & 2    &      & 1 \\
        & 0 &   & 2    &      & 0 \\
        &   & 0 &      & 0 \\
        &   &   & 1 \\
    \end{array}.
  \end{equation}
\end{example}

\begin{example}
  \label{example:Z2xZ2-threefold}
  Another example of a hyperelliptic threefold
  satisfying the condition of \cref{lemma:regular-hyperelliptic-threefold},
  which is moreover \emph{abelian},
  is the following.
  Let~$G=\mathbb{Z}/2\mathbb{Z}\times\mathbb{Z}/2\mathbb{Z}$
  act on the product~$E_1\times E_2\times E_3$ of elliptic curves,
  by
  \begin{equation}
    \begin{aligned}
      g_1(z_1,z_2,z_3)&=\left(-z_1,-z_2+\frac{1}{2},z_3+\frac{1}{2}\right) \\
      g_2(z_1,z_2,z_3)&=\left(z_1+\frac{1}{2},-z_2,-z_3\right)
    \end{aligned}
  \end{equation}
  where we write~$E_i$ as the quotient of~$\mathbb{C}$
  by a lattice of the form~$\mathbb{Z}+\tau\mathbb{Z}$.
  Its Hodge diamond is
  \begin{equation}
    \begin{array}{ccccccc}
        &   &   & 1 \\
        &   & 0 &      & 0 \\
        & 0 &   & 3    &      & 0 \\
      1 &   & 3 &      & 3    &      & 1 \\
        & 0 &   & 3    &      & 0 \\
        &   & 0 &      & 0 \\
        &   &   & 1 \\
    \end{array}.
  \end{equation}
\end{example}

We do not know whether a regular hyperelliptic threefold,
such as the ones in \cref{example:D4-threefold,example:Z2xZ2-threefold}, is stably indecomposable.

\begin{remark}
  \label{remark:no-exceptional-objects}
  The argument in the proof of \Cref{lemma:regular-hyperelliptic-threefold} shows that
  line bundles are not exceptional objects in the derived category of a hyperelliptic variety,
  providing some further evidence for \Cref{conjecture:indecomposability}, albeit a small one.
  This generalizes to higher dimensions one of the key differences between derived categories
  of bielliptic surfaces and Enriques surfaces.
\end{remark}

\paragraph{Algebraic triviality of the canonical bundle}
The first proof of indecomposability for bielliptic surfaces used the following criterion \cite[Corollary~1.7]{1508.00682v2}.
\begin{theorem}[Kawatani--Okawa]
  \label{theorem:kawatani-okawa}
  Let~$X$ be a smooth projective variety
  for which~$[\omega_X]\in\Pic^0(X)$.
  Then~$\derived^\bounded(X)$ is indecomposable.
\end{theorem}
The proof of indecomposability for bielliptic surfaces \cite[Proposition~4.1]{1508.00682v2}
subsequently uses that a bielliptic surface~$S$
has an elliptic fibration without multiple fibers as Albanese morphism
and that~$\omega_S\cong\alb^*\mathcal{L}$ for a torsion line bundle on~$\Alb(S)$
(see also, e.g.,~\cite[\S4.D]{MR0565468})
to conclude that~$\omega_S\in\Pic^0S$.
In the following proposition we generalize this to higher-dimensional hyperelliptic varieties
whose Albanese morphism has Calabi--Yau fibers.

\begin{proposition}
  \label{proposition:canonical-bundle-formula}
  Let~$X$ be a hyperelliptic variety whose Albanese fibers are Calabi--Yau varieties,
  i.e., their canonical bundle is trivial.
  Then
  \begin{equation}
    \label{equation:cy-fibers}
    \omega_X\cong\alb_{X}^*\mathcal{L}
  \end{equation}
  with~$\mathcal{L}$ a torsion line bundle on $\Alb{X}$.
  In particular, $[\mathcal{L}] \in \operatorname{Pic}^{0}(\Alb{X})$ and thus $[\omega_{X}] \in \operatorname{Pic}^{0}(X)$.
\end{proposition}

\begin{proof}
  Let us write $C \colonequals \Alb(X)$
  (previously denoted~$B$, but the letter~$B$ will play a standard role
  in the canonical bundle formula we will use shortly)
  and $f \colonequals \operatorname{alb}_{X}$ to shorten the notation.
  By assumption, the general fiber $F$ of $f$ has $\omega_{F} \cong \mathcal{O}_{F}$.
  Therefore, $\hh^{0}(\omega_{F}) = 1$,
  so~$1$ is the smallest~$b \in \mathbb{Z}_{>0}$
  such that the~$b$th plurigenus~$P_{b}(F)$ is non-zero.

  Moreover, since $C$ is an abelian variety, we also have $\omega_{C} \cong \mathcal{O}_{C}$.
  Therefore, $\mathrm{K}_{X/C} = \mathrm{K}_{X}$.

  The canonical bundle formula \cite[Proposition 2.2]{MR1863025}
  implies that there exist a $\mathbb{Q}$-divisor~$D$ on~$C$
  and a $\mathbb{Q}$-divisor~$B$ on~$X$ with the following properties:
  \begin{enumerate}
    \item There exists a graded $\mathcal{O}_{C}$-algebra isomorphism
      \begin{equation}
        \label{equation:canonical-bundle-formula-algebra}
        \bigoplus_{i \geq 0} \mathcal{O}_C(\lfloor iD \rfloor)
        \cong
        \bigoplus_{i \geq 0} (f_{*}\mathcal{O}_X(i\mathrm{K}_{X}))^{\vee\vee},
      \end{equation}
      where $(-)^{\vee\vee}$ denotes the reflexive hull.
    \item The isomorphism \eqref{equation:canonical-bundle-formula-algebra} induces an equality
      \begin{equation}
        \label{equation:cbf2}
        \mathrm{K}_{X} = f^{*}(D) + B,
      \end{equation}
      where $B$ is a $\mathbb{Q}$-divisor on $X$,
      which decomposes as~$B = B_{+} - B_{-}$,
      where~$B_{+}$ and $B_{-}$ are effective~$\mathbb{Q}$-divisors
      without common irreducible components,
      such that $f_{*}\mathcal{O}_{X}(\lfloor iB_{+} \rfloor) = \mathcal{O}_{C}$
      for all $i > 0$ and $\operatorname{codim}(f(\operatorname{Supp}{B_{-}})) \geq 2$.
  \end{enumerate}
  Since $f$ is equidimensional by \cref{lemma:fiber-bundle-structure},
  we have $B_{-} = 0$ and $B$ is itself an effective $\mathbb{Q}$-divisor.

  Indeed, suppose that~$E$ is a prime Weil divisor on $X$
  such that~$\operatorname{codim}(f(E)) \geq 2$.
  Let $m \colonequals \dim(X)$, $n \colonequals \dim(C)$ and $d \colonequals m - n$.
  By equidimensionality,~$d$ is the dimension of any fiber of $f$.
  The assumption on $E$ implies that $Z \colonequals f(E)$ has dimension at most $n-2$,
  and since $E \subseteq f^{-1}(Z)$, $W \colonequals f^{-1}(Z)$ has dimension at least $m - 1$.
  Consider now the morphism $f|_{W} \colon W \to Z$
  and let $F$ be a general fiber of this morphism,
  so that $\dim(F) = \dim(W) - \dim(Z) \geq m - n + 1 > d$.
  Then $F$ is also a fiber of $f$, so its dimension should be $d$, a contradiction.
  Therefore $B_{-} = 0$ and $B = B_{+}$ is an effective $\mathbb{Q}$-divisor.

  Moreover, it follows from the proof of \cite[Theorem 1.12]{MR0927961}
  and the smoothness of $f$
  that $D$ is a divisor with integer coefficients,
  because the constant $c$ in the proof of \cite[Proposition 2.2]{MR1863025} can be taken to be~$1$.
  Therefore, $B$ is also a divisor with integer coefficients.

  Let $i > 0$ be such that $i\mathrm{K}_{X} \sim 0$.
  Then from \eqref{equation:canonical-bundle-formula-algebra}
  and the isomorphism~$\mathcal{O}_{C} \cong f_{*}\mathcal{O}_{X}$
  we deduce that~$\mathcal{O}_{C}(iD) \cong \mathcal{O}_{C}$, i.e.,~$iD \sim 0$.
  It follows now from \eqref{equation:canonical-bundle-formula-algebra}
  that $iB = i\mathrm{K}_{X} - f^{*}(iD) \sim 0$ as well.
  So both $D$ and $B$ are torsion divisors.

  In particular, $B$ is an effective divisor which is torsion.
  If $B \neq 0$, then it can be written as the zero locus of a global section of the torsion line bundle $\mathcal{O}(B)$.
  The existence of such a section would imply that~$\mathcal{O}_X(B) \cong \mathcal{O}_X$,
  so we would deduce that $B \sim 0$.
  So in any case, we deduce that
  \begin{equation}
    \label{equation:KX-is-pullback}
    \mathrm{K}_{X} \sim f^{*}(D)
  \end{equation}
  for some torsion divisor $D$ on $C$.
  Since $C$ is an abelian variety,
  $\HH^{2}(C,\mathbb{Z})$ is torsion-free by \cite[Corollary~1.3.3]{MR2062673},
  so numerically trivial divisors are algebraically trivial.
  Therefore, $[\mathcal{O}_{C}(D)] \in \operatorname{Pic}^{0}(C)$.
  Finally, pullback preserves the identity component $\operatorname{Pic}^{0}$,
  and thus~$[\omega_{X}] \in \operatorname{Pic}^{0}(X)$ as well.
\end{proof}

This gives us the following corollary.
\begin{corollary}
  \label{corollary:cy-fibers}
  Let $X$ be a hyperelliptic variety whose Albanese fibers are Calabi--Yau varieties.
  Then $\derived^\bounded(X)$ is indecomposable.
\end{corollary}

\begin{proof}
  This follows from \Cref{proposition:canonical-bundle-formula} and \Cref{theorem:kawatani-okawa}.
\end{proof}

\begin{remark}
  \label{remark:necessary-condition}
  The assumption that the fiber $F$ is a Calabi--Yau variety in \Cref{proposition:canonical-bundle-formula}
  is a necessary condition for the conclusion that the canonical bundle of $X$
  is the pull-back of a line bundle on the Albanese.
  Indeed, recall from \Cref{remark:torsion-order-of-fiber} that $\mathrm{K}_{X}|_{F} = \mathrm{K}_{F}$,
  so if $\omega_{X}$ is the pull-back of a line bundle on the Albanese,
  then $\omega_{F}$ is the pull-back of a line bundle on a point.
  \Cref{proposition:canonical-bundle-formula} shows that this assumption
  is also a sufficient condition in this setting.
\end{remark}

\begin{remark}
  \label{remark:paracanonical-base-locus}
  If the Albanese fibers of $X$ are Calabi--Yau, then \Cref{proposition:canonical-bundle-formula} shows
  that the paracanonical base locus is empty.
  Conversely, if the Albanese fibers of $X$ are not Calabi--Yau,
  then the paracanonical base locus is all of $X$.
  Indeed, in this case the Albanese fibers are hyperelliptic varieties
  whose canonical bundle is torsion but not trivial,
  so $\HH^{0}(X_{b},\omega_{X}|_{X_{b}}) = \HH^{0}(X_{b},\omega_{X_{b}}) = 0$
  for all closed points $b \in \Alb(X)$.
  By Grauert's theorem \cite[Corollary III.12.9]{MR0463157},
  $(\alb_{X})_{*}\omega_{X} = 0$, so the relative base locus with respect to~$\alb_{X}$
  is the whole $X$.
  By \cite[Theorem 1.5]{MR4805086}, the paracanonical base locus of $X$ is also all of $X$.
  In particular, Lin's theorem \cite[Corollary 1.6]{2107.09564v3} applies
  to a hyperelliptic variety $X$ if and only if \Cref{proposition:canonical-bundle-formula} does,
  and thus this approach to proving semiorthogonal indecomposability
  does not give any new information.
\end{remark}

\begin{remark}
  \label{remark:cy-fibers}
  In the case of a cyclic hyperelliptic variety $X$ satisfying the assumption in \Cref{proposition:cyclic-implies-abelian-fibers},
  \Cref{proposition:canonical-bundle-formula} gives an alternative argument
  to show the indecomposability of $\derived^\bounded(X)$.
  From \Cref{proposition:canonical-bundle-formula} we deduce moreover that $[\omega_{X}] \in \operatorname{Pic}^{0}(X)$,
  but on the downside, using this argument we cannot deduce stable indecomposability.
\end{remark}

\begin{remark}
  \label{remark:conditions-diverge}
  As explained in the introduction, all three conditions in \Cref{proposition:three-results} apply to bielliptic surfaces.
  In higher dimensions, having irregularity~$q_{X} = \dim X - 1$ implies being cyclic by \cref{proposition:large-irregularity-is-cyclic}.
  However, in higher dimensions, one can find:
  \begin{itemize}
    \item non-cyclic hyperelliptic varieties of irregularity~$q_X=\dim X-2$
      whose Albanese fiber has non-trivial canonical bundle (see \cref{example:hyperelliptic-fibers});
    \item cyclic hyperelliptic varieties
      whose Albanese fiber has non-trivial canonical bundle (see \cref{example:bielliptic-fibers});
    \item hyperelliptic varieties whose Albanese fiber has trivial canonical bundle,
      yet which are not cyclic and have~$q_X\leq\dim X-3$ (see \cref{example:D4-threefold,example:Z2xZ2-threefold});
    \item cyclic hyperelliptic varieties of irregularity~$q_X\leq\dim X-3$
      (see \cref{remark:low-irregularity-cyclic}).
  \end{itemize}
\end{remark}

\renewcommand*{\bibfont}{\small}
\printbibliography

\emph{Pieter Belmans}, \url{pieter.belmans@uni.lu} \\
Department of Mathematics, Universit\'e de Luxembourg, 6, avenue de la Fonte, L-4364 Esch-sur-Alzette, Luxembourg

\emph{Andreas Demleitner}, \url{andreas.demleitner@math.uni-freiburg.de} \\
Mathematical Institute, University of Freiburg, Ernst-Zermelo-Str. 1, 79104 Freiburg im Breisgau, Germany

\emph{Pedro N\'u\~nez}, \url{pnunez@ntu.edu.tw} \\
Department of Mathematics, National Taiwan University, No. 1, Sec. 4, Roosevelt Rd., Taipei 10617, Taiwan

\end{document}